\documentclass{amsart}

\usepackage{amsfonts}
\usepackage{amssymb}
\usepackage{enumerate}
\usepackage{amsmath}
\usepackage{amsthm}
\usepackage{graphicx}
\usepackage[english]{babel}

\numberwithin{equation}{section}

\newtheorem{theorem}{Theorem}[section]
\newtheorem{lemma}[theorem]{Lemma}

\newtheorem{question}[theorem]{Question}

\newtheorem{corollary}[theorem]{Corollary}

\newtheorem{fact}[theorem]{Fact}
\newtheorem*{alef}{Theorem \ref{t:alef}}
\newtheorem*{x}{Theorem \ref{mainthm}}
\newtheorem*{g}{Theorem \ref{t:gauge}}
\newtheorem*{y}{Theorem \ref{t:dtn}}
\newtheorem*{c1}{Corollary \ref{c:main1}}
\newtheorem*{f1}{Theorem \ref{t:fractal1}}
\newtheorem*{f2}{Corollary \ref{c:fractal2}}
\newtheorem*{b1}{Theorem \ref{t:bound}}
\newtheorem*{m}{Theorem \ref{t:max}}
\newtheorem*{ex}{Theorem \ref{t:ex}}
\newtheorem*{ex2}{Fact \ref{f:ex2}}
\newtheorem*{cm}{Corollary \ref{c:maxtop}}

\newtheorem*{alt}{Corollary \ref{c:alt}}

\theoremstyle{definition}

\newtheorem{remark}[theorem]{Remark}
\newtheorem{definition}[theorem]{Definition}
\newtheorem{notation}[theorem]{Notation}

\DeclareMathOperator{\diam}{diam}
\DeclareMathOperator{\dist}{dist}

\DeclareMathOperator{\inter}{int}

\DeclareMathOperator{\cl}{cl}

\DeclareMathOperator{\pr}{pr}

\newcommand{\NN}{\mathbb{N}}

\newcommand{\RR}{\mathbb{R}}
\newcommand{\ZZ}{\mathbb{Z}}
\newcommand{\eps}{\varepsilon}

\newcommand{\iA}{\mathcal{A}}
\newcommand{\iB}{\mathcal{B}}
\newcommand{\iC}{\mathcal{C}}
\newcommand{\iD}{\mathcal{D}}
\newcommand{\iG}{\mathcal{G}}
\newcommand{\iH}{\mathcal{H}}

\newcommand{\iK}{\mathcal{K}}
\newcommand{\iL}{\mathcal{L}}

\newcommand{\iF}{\mathcal{F}}
\newcommand{\iP}{\mathcal{P}}

\newcommand{\iN}{\mathcal{N}}
\newcommand{\iU}{\mathcal{U}}

\newcommand{\iV}{\mathcal{V}}
\newcommand{\iW}{\mathcal{W}}

\newcommand{\iQ}{\mathcal{Q}}

\begin{document}

\title{Dimensions of fibers of generic continuous maps}
\author{Rich\'ard Balka}
\address{Department of Mathematics, University of British Columbia, and Pacific Institute for the Mathematical Sciences, Vancouver, BC V6T 1Z2, Canada}
\email{balka@math.ubc.ca}
\thanks{Supported by the National Research, Development and Innovation Office--NKFIH, 104178.}



\begin{abstract} In an earlier paper Buczolich, Elekes, and the author described the Hausdorff dimension of
the level sets of a generic real-valued continuous function (in the sense of Baire
category) defined on a compact metric space $K$ by introducing
the notion of topological Hausdorff dimension. Later on, the
author extended the theory for maps from $K$ to $\mathbb{R}^n$.
The main goal of this paper is to generalize the relevant results for topological and packing dimensions
and to obtain new results for sufficiently homogeneous spaces $K$ even in the case case of Hausdorff dimension.

Let $K$ be a compact metric space and let us denote by $C(K,\mathbb{R}^n)$ the set of continuous maps from $K$ to $\mathbb{R}^n$ endowed with the maximum norm. Let $\dim_{*}$ be one of the topological dimension $\dim_T$, the Hausdorff dimension $\dim_H$, or the packing dimension $\dim_P$. Define
\[d_{*}^n(K)=\inf\{\dim_{*}(K\setminus F): F\subset K \textrm{ is $\sigma$-compact with } \dim_T F<n\}.\]
We prove that $d^n_{*}(K)$ is the right notion to describe the
dimensions of the fibers of a generic continuous map $f\in C(K,\mathbb{R}^n)$.
In particular, we show that $\sup\{\dim_{*}f^{-1}(y): y\in \mathbb{R}^n\} =d^n_{*}(K)$
provided that $\dim_T K\geq n$, otherwise every fiber is finite.
Proving the above theorem for packing dimension requires entirely new ideas.
Moreover, we show that the supremum is attained on the left hand side of the above equation.

Assume $\dim_T K\geq n$. If $K$ is sufficiently homogeneous, then we can say much more.
For example, we prove that $\dim_{*}f^{-1}(y)=d^n_{*}(K)$ for a generic $f\in C(K,\mathbb{R}^n)$ for all
$y\in \inter f(K)$ if and only if $d^n_{*}(U)=d^n_{*}(K)$ or $\dim_T U<n$ for all open sets $U\subset K$.
This is new even if $n=1$ and $\dim_{*}=\dim_H$. It is known that for a generic $f\in C(K,\mathbb{R}^n)$ the interior of $f(K)$ is not empty.
We augment the above characterization by showing that $\dim_T \partial f(K)=\dim_H \partial f(K)=n-1$ for a generic $f\in C(K,\mathbb{R}^n)$.
In particular, almost every point of $f(K)$ is an interior point.

In order to obtain more precise results, we use the concept of
generalized Hausdorff and packing measures, too.
\end{abstract}

\maketitle

\section{Introduction}
Over the last three decades there has been a large interest in determining the various fractal dimensions of graphs, images, and fibers of generic continuous maps in the sense of Baire category. 
In the case of graphs see the papers of Mauldin and Williams~\cite{MW}, Balka, Buczolich, and Elekes~\cite{BBE2} for Hausdorff dimension, Humke and Petruska~\cite{HP}, Liu, Tan, and Wu \cite{LTW} for packing dimension, and Hyde, Laschos, Olsen, Petrykiewicz, and Shaw~\cite{HLOPS} for box dimensions. Dimensions of images of generic continuous maps were determined by 
Balka, Farkas, Fraser, and Hyde~\cite{BFFH}.

This paper considers the case of fibers. Let $\dim_T$, $\dim_H$, and $\dim_P$ denote the topological, Hausdorff, and packing dimension, respectively. We adopt the convention $\dim \emptyset=-1$ for all dimensions. 

\subsection{The Hausdorff dimension of fibers} In this subsection we present known results about the Hausdorff dimension of fibers of generic maps and connect them to our terminology. For a compact metric space $K$ and $n\in \NN^+$ let $C(K,\RR^n)$ denote the set of continuous maps from $K$ to $\mathbb{R}^n$ equipped with the maximum norm. Since this is a complete metric space, we can use Baire category arguments. 

The following theorem is folklore. 
\begin{theorem} For a generic $f\in C([0,1],\RR)$ for all $y\in f([0,1])$ we have
\[\dim_H f^{-1}(y)=0.\]
\end{theorem}

The following higher dimensional version is due to Kirchheim \cite[Theorem~2]{BK}, which requires more elaborate arguments.

\begin{theorem}[Kirchheim] \label{t:Kirchheim} Let $m,n\in \mathbb{N}^+$ with $m\geq n$. For a generic continuous map
$f\in C([0,1]^m,\mathbb{R}^n)$ for all $y\in \inter f\left([0,1]^m\right)$ we have
\[\dim_H f^{-1}(y)=m-n.\]
\end{theorem}

In order to generalize Kirchheim's theorem for 
$C(K,\RR)$, Balka,
Buczolich, and Elekes~\cite{BBE} introduced a new dimension for metric spaces, the topological Hausdorff dimension. First recall the definition of topological dimension.

\begin{definition} Set $\dim _{T} \emptyset = -1$. The \emph{topological dimension}
of a non-empty metric space $X$ is defined by induction as
\[
\dim_{T} X=\inf\{d: X \textrm{ has a basis } \mathcal{U}  \textrm{ such that }
\dim_{T} \partial {U} \leq d-1 \textrm{ for all } U\in \mathcal{U} \}.
\]
\end{definition}

The topological Hausdorff dimension is defined analogously. However, it is not inductive, and it can attain non-integer values as well.

\begin{definition}\label{deftoph}
Set $\dim _{TH} \emptyset=-1$. The \emph{topological Hausdorff
dimension} of a non-empty metric space $X$ is defined as
\[
\dim_{TH} X=\inf\{d:
 X \textrm{ has a basis } \mathcal{U} \textrm{ such that } \dim_{H} \partial {U} \leq d-1 \textrm{ for all }
 U\in \mathcal{U} \}.
\]
\end{definition}

Then next theorem is due to Balka, Buczolich, and Elekes~\cite[Theorem~7.12]{BBE}. 

\begin{theorem}[Balka--Buczolich--Elekes] \label{t:BBE} 
Let $K$ be a compact metric space with $\dim_{T} K\geq 1$. Then for a generic $f\in C(K,\RR)$ we have
\begin{enumerate}[(i)]
\item $\dim_{H} f^{-1} (y)\leq \dim_{TH} K-1$ for all $y\in \mathbb{R}$,
\item for every $d<\dim_{TH} K$ there exists a non-degenerate interval $I_{f,d}$ such that $\dim_{H} f^{-1} (y)\geq d-1$ for every $y\in I_{f,d}$.
\end{enumerate}
\end{theorem}

In order to generalize the above theorem to higher dimensions, we consider the following inductive definition due to Darji and Elekes \cite{DE}. The topological Hausdorff dimension corresponds to the case $n=1$.

\begin{definition} Let $\dim^{n}_{TH} \emptyset=-1$ for all $n\in \mathbb{N}$. Let $\dim^{0}_{TH}X=\dim_H X$ for each non-empty metric space $X$.
The \emph{$n$th inductive topological Hausdorff dimension} is defined inductively as
\[\dim^{n}_{TH} X=\inf \left\{ d: X \textrm{ has a basis } \mathcal{U} \textrm{ such that }
\dim^{n-1}_{TH} \partial {U} \leq d-1 \textrm{ for all } U\in \mathcal{U} \right\}.\]
\end{definition}

All above notions of dimension can attain the value $\infty$ as well,
we adopt the convention $\infty-1 = \infty$, hence $d = \infty$ is a member of the
above sets. For more information see \cite{B}. Here we extend the definition of $d_{*}^{n}$ to separable metric spaces.
\begin{definition} \label{d:dn} Let $n\in \NN^+$ and let $X$ be a separable metric space. Let $\dim_{*}$ be one of $\dim_T$, $\dim_H$, or $\dim_P$.
Then
\[d^n_{*}(X)=\inf \{\dim_{*} (X\setminus F): F\subset X \textrm{ is an } F_{\sigma}\textrm{ set with } \dim_T F\leq n-1\}.\]
\end{definition}

As Hausdorff dimension admits $G_{\delta}$ hulls, \cite[Theorem~4.4]{B} yields the following.

\begin{theorem} \label{t:equ} If $n\in \NN^+$ and $X$ is a separable metric space with $\dim_T X\geq n$ then
\[d^n_H(X)=\inf\{\dim_{H}(X\setminus A): A\subset X \textrm{ with } \dim_T A<n\}=\dim^n_{TH} X-n.\]
\end{theorem}

Assume that $K$ is a compact metric space and $n\in \NN^+$. 
If $\dim_T K<n$ then the fibers of a generic $f\in C(K,\RR^n)$ are all finite,
see Theorem~\ref{t:Hurewicz}.

From now on suppose that $\dim_T K\geq n$. By Theorem~\ref{t:equ}
the main theorem of \cite{B} can be formalized as follows, which extends Theorem~\ref{t:BBE}.
\begin{theorem}[Balka] \label{t:B} 
Let $n\in \mathbb{N}^+$ and assume that $K$ is a compact metric space with $\dim_{T}
K\geq n$. Then for a generic $f\in C(K,\RR^n)$ we have
\begin{enumerate}[(i)]
\item $\dim_{H} f^{-1} (y)\leq d^n_H(K)$ for all $y\in \mathbb{R}^n$,
\item for every $d<d^n_H(K)$ there exists a non-empty
open set $U_{f,d}\subset \mathbb{R}^n$ such that $\dim_{H} f^{-1} (y)\geq d$ 
for every $y\in U_{f,d}$.
\end{enumerate}
\end{theorem}

\subsection{The Main Theorem} We generalize Theorem~\ref{t:B} for topological and packing dimensions as well, which is the main result of our paper.

\begin{x}[Main Theorem, simplified version] Let $n\in \NN^+$ and let $K$ be a compact metric space with $\dim_T K\geq n$.
Let $\dim_{*}$ be one of $\dim_T$, $\dim_H$, or $\dim_P$. Then for a generic $f\in C(K,\RR^n)$ we have
\begin{enumerate}[(i)]
\item \label{m1} $\dim_{*} f^{-1} (y)\leq d^n_{*}(K)$ for all $y\in \RR^n$,
\item for every $d<d^n_{*}(K)$ there exists a non-empty open set $U_{f,d}\subset \RR^n$ such that $\dim_{*} f^{-1}(y)\geq d$ for all $y\in U_{f,d}$.
\end{enumerate}
\end{x}

\begin{c1}
Let $n\in \NN^+$ and let $K$ be a compact metric space with $\dim_T K\geq n$. Let $\dim_{*}$ be one of $\dim_T$, $\dim_H$, or $\dim_P$.
For a generic $f\in C(K,\RR^n)$ we have
\[\sup\{\dim_{*}f^{-1}(y): y\in \mathbb{R}^n\}=d^n_{*}(K).\]
\end{c1}

In Section~\ref{s:main} we prove Theorem~\ref{mainthm}~\eqref{m1}. In Subsection~\ref{ss:top} we show the Main Theorem for topological dimension. 
We simplify the proof of Theorem~\ref{t:B} by using the rich theory of topological dimension and the following formula for $d_T^{n}(K)$. 

\begin{y} If $n\in \NN^+$ and $X$ is a separable metric space with $\dim_T X \geq n$ then
	\[d_{T}^n (X)=\dim_T X-n.\]
\end{y}

In Subsection~\ref{ss:packing} we prove the Main Theorem for packing dimension. We need to apply different methods,
because the packing dimension, unlike the topological and Hausdorff dimensions, does not admit $G_{\delta}$ hulls.
In fact, we prove a stronger result using the concept of generalized packing measure $\iP^h$, see Section~\ref{s:pre} for the definition.

\begin{g} Let $n\in \NN^+$ and let $h$ be a left-continuous gauge function.
Let $K$ be a compact metric space such that $\dim_T U\geq n$ and $\iP^h(U)=\infty$ for all
non-empty open sets $U\subset K$. Then for a generic $f\in C(K,\RR^n)$ for all $y\in \inter f(K)$ we have
\[ \iP^h(f^{-1}(y))=\infty.\]
Moreover, the measures $\iP^h|_{f^{-1}(y)}$ are not $\sigma$-finite.
\end{g}

The following theorem characterizes the infinite
fibers of a generic $f\in C(K,\RR^n)$.
Its first part follows from a result of Kato~\cite[Theorem~4.6]{K}, while the second part is an easy corollary of Theorem~\ref{t:gauge}.

\begin{alef}
Let $n\in \NN^+$ and let $K$ be a compact metric space with $\dim_T K\geq n$. Then for a generic
$f\in C(K,\RR^n)$ we have
\begin{enumerate}[(i)]
\item  $\#f^{-1}(y)\leq n$ if $y\in \partial f(K)$,
\item $\# f^{-1}(y)=2^{\aleph_0}$ if $y\in \inter f(K)$.
\end{enumerate}
\end{alef}

\subsection{Fibers on fractals} If $K$ is sufficiently homogeneous, our Main Theorem can be generalized. The following
theorem is \cite[Theorem~7.9]{B}. The case $n=1$ is due to Balka, Buczolich, and Elekes~\cite[Theorem~3.6]{BBE2}.

\begin{theorem}[Balka] \label{t:Bss}
	Let $n\in \NN^+$ and let $K$ be a self-similar compact metric space with $\dim_{T} K\geq n$.
	For a generic $f\in C(K,\RR^n)$ for any $y\in \inter f(K)$ we have
	\[ \dim_{H} f^{-1}(y)=d_H^{n}(K).\]
\end{theorem}

In Section~\ref{s:fractal} we characterize the compact sets for which the analogues of the above statement hold. This is new even in the case of Hausdorff dimension and $n=1$. Let $\dim_{*}$ be one of $\dim_T$, $\dim_H$, or $\dim_P$. For each $f\in C(K,\RR^n)$ let
\begin{equation*} R_{*}(f)=\{y\in f(K): \dim_{*} f^{-1}(y)=d^n_{*}(K)\}.
\end{equation*}

If $d_{*}^n(K)=0$ then the Main Theorem yields that $\dim_{*} f^{-1}(y)=d^n_{*}(K)=0$ for a generic $f\in C(K,\RR^n)$ for all
$y\in f(K)$, so we may assume that $d_{*}^n(K)>0$.

\begin{f1} Let $n\in \NN^+$ and assume that $\dim_{*}$ is one of $\dim_T$, $\dim_H$, or $\dim_P$.
Let $K$ be a compact metric space with $d^n_{*}(K)>0$. The following are equivalent:
\begin{enumerate}[(i)]
\item $R_{*}(f)=\inter f(K)$ for a generic $f\in C(K,\RR^n)$;
\item $R_{*}(f)$ is dense in $\inter f(K)$ for a generic $f\in C(K,\RR^n)$;
\item $d^n_{*}(U)=d^n_{*}(K)$ or $\dim_T U<n$ for all open sets $U\subset K$.
\end{enumerate}
\end{f1}

Since $\dim_T K\geq n$, we have $\inter f(K)\neq \emptyset$ for a generic $f\in C(K,\RR^n)$, see
Theorem~\ref{t:int}. Thus the first statement of the above theorem is never vacuous.

We obtain the following corollary. The extra assumption is
that each non-empty open subset of $K$ has topological dimension at least $n$,
which makes $\inter f(K)$ dense in $f(K)$ for a generic $f\in C(K,\RR^n)$.

\begin{f2}
Let $n\in \NN^+$ and assume that $\dim_{*}$ is one of $\dim_T$, $\dim_H$, or $\dim_P$.
Let $K$ be a compact metric space with $d^n_{*}(K)>0$. The following are equivalent:
\begin{enumerate}[(1)]
\item $R_{*}(f)=\inter f(K)$ and $\inter f(K)$ is dense in $f(K)$ for a generic $f\in C(K,\RR^n)$;
\item $R_{*}(f)$ is dense in $f(K)$ for a generic $f\in C(K,\RR^n)$;
\item  $d^n_{*}(U)=d^n_{*}(K)$ for all non-empty open sets $U\subset K$.
\end{enumerate}
\end{f2}

In particular, we obtain an extension of Kirchheim's theorem for topological and packing dimensions, see Corollary~\ref{c:Ki}.

\subsection{The boundary of generic images} Kirchheim~\cite[Theorem~2]{BK} proved the following result, which augments Theorem~\ref{t:Kirchheim}.

\begin{theorem}[Kirchheim] Let $m\geq n$ be positive integers.
Then for a generic $f\in C([0,1]^m,\RR^n)$ we have
\[ \dim_H \partial f(K)=n-1.\]
\end{theorem}

In Section~\ref{s:partial} we generalize this theorem by replacing $[0,1]^m$ with an arbitrary compact metric space $K$ satisfying
$\dim_T K\geq n$. We  denote by $\iH^h$ the generalized Hausdorff measure, for the definition see Section~\ref{s:pre}.  

\begin{b1}Let $n \in \NN^+$ and let $K$ be a compact metric space with $\dim_T K\geq n$.
Then for a generic $f\in C(K,\RR^n)$ we have
\[\dim_T \partial f(K)=\dim_H \partial f(K)=n-1.\]
Moreover, let $h$ be a gauge function with $\lim_{r\to 0+} h(r)/r^{n-1}=0$.
Then for a generic $f\in C(K,\RR^n)$ we have $\iH^h(\partial f(K))=0$ and $\iH^{n-1}(\partial f(K))>0$.
\end{b1}

For each $f\in C(K,\RR^n)$ let
\begin{equation*} S_{*}(f)=\{y\in f(K): \dim_{*} f^{-1}(y)<d^n_{*}(K)\}.
\end{equation*}

Theorems~\ref{t:fractal1} and~\ref{t:bound} imply the following.

\begin{alt} Let $n\in \NN^+$ and assume that $\dim_{*}$ is one of $\dim_T$, $\dim_H$, or $\dim_P$.
Let $K$ be a compact metric space with $d_{*}^n(K)>0$. Exactly one of the following holds:
\begin{enumerate}[(a)]
\item \label{a1} $\dim_H S_{*}(f)=n-1$ for a generic $f\in C(K,\RR^n)$;
\item  $\inter S_{*} (f)\neq \emptyset$ for a generic $f\in C(K,\RR^n)$.
\end{enumerate}
Moreover, \eqref{a1} is equivalent to the statements of Theorem~\ref{t:fractal1}.
\end{alt}

\subsection{Fibers of maximal dimension}
In Section~\ref{s:max} we show that the supremum is attained in Corollary~\ref{c:main1}.

\begin{m} Let $n\in \NN^+$ and let $K$ be a compact metric space with $\dim_T K\geq n$.
Let $\dim_{*}$ be one of $\dim_T$, $\dim_H$, or $\dim_P$. For a generic $f\in C(K,\RR^n)$ we have
\[\max\{\dim_{*}f^{-1}(y): y\in \RR^n\}=d^n_{*}(K).\]
\end{m}

If $\dim_{*}=\dim_T$ and $\dim_T K<\infty$ then the Main Theorem yields more.

\begin{cm} Assume that $n\in \NN^+$ and $K$ is a compact metric space satisfying $n\leq \dim_T K<\infty$.
Then for a generic $f\in C(K,\RR^n)$ there is a non-empty open set $U_f\subset \RR^n$ such that for all $y\in U_f$ we have
\[\dim_T f^{-1}(y)=d_{T}^n(K).\]
\end{cm}

We show that otherwise Theorem~\ref{t:max} is sharp in general.
For $n=1$ the Hausdorff dimension part of the next theorem follows from \cite[Theorem~4.5]{BBE2}.

\begin{ex} For each $n\in \NN^+$ there is a compact set $K\subset \RR^{n+1}$ such that for each $f\in C(K,\RR^n)$ there is a
$y_f\in \RR^n$ such that
\begin{enumerate}[(i)]
\item $d_H^n(K)=1$ and $d_{P}^{n}(K)=n+1$,
\item $\dim_H f^{-1}(y)<1$ for a generic $f\in C(K,\RR^n)$ for all $y\in \RR^n\setminus \{y_f\}$,
\item $\dim_P f^{-1}(y)<n+1$ for every $f\in C(K,\RR^n)$ for all $y\in \RR^n\setminus \{y_f\}$.
\end{enumerate}
\end{ex}

\begin{ex2} There is a compact metric space $K$ such that
$\dim_T K=\infty$, and for each $f\in C(K,\RR^n)$ there is a
$y_f\in \RR^n$ such that $\dim_{T} f^{-1}(y)<\infty$ for all $y\in \RR^n\setminus \{y_f\}$.
\end{ex2}

\subsection{An open question} We are curious whether Theorem~\ref{t:bound} can be extended.

\begin{question} Let $n\geq 2$ be a positive integer and let $K$ be a compact metric space with $\dim_T K\geq n$. 
	Is it true that for a generic continuous map $f\in C(K,\RR^n)$ we have $\dim_P \partial f(K)=n$ and 
	the measure $\iH^{n-1}|_{\partial f(K)}$ is not $\sigma$-finite? 
\end{question}

\section{Preliminaries} \label{s:pre}

Let $(X,\rho)$ be a metric space, and let $A,B\subset X$ be arbitrary
sets. We denote by $\cl A$, $\inter A$, and $\partial A$ the closure,
interior, and boundary of $A$, respectively.
The diameter of $A$ is denoted by $\diam A$. We use
the convention $\diam \emptyset = 0$. The distance of the sets $A$
and $B$ is defined by $\dist (A,B)=\inf \{\rho(x,y): x\in A, \, y\in
B\}$. Let $B(x,r)=\{y\in X: \rho(x,y)\leq r\}$, and $U(x,r)=\{y\in X:
\rho(x,y)< r\}$. We say that $A$ is a \emph{regular closed} set if $A=\cl (\inter A)$.
The set $A$ is a regular closed set iff there is an open set $U\subset X$ such that $A=\cl U$.

\begin{fact} \label{f:reg} Let $X_0$ be a metric space and let $k\in \NN^+$. Let $X_k\subset X_{k-1}\subset \dots \subset X_0$ such that
$X_k\neq \emptyset $ and $X_{i}$ is a regular closed set in the
relative topology of $X_{i-1}$ for all $i\in \{1,\dots, k\}$. Then $X_k$ contains a non-empty open subset of $X_0$.
\end{fact}

Let $X$ be a \emph{complete} metric space. A set is \emph{somewhere
dense} if it is dense in a non-empty open set, otherwise it is
called \emph{nowhere dense}. We say that $M \subset X$ is
\emph{meager} if it is a countable union of nowhere dense sets, and
a set is of \emph{second category} if it is not meager. A set is
called \emph{co-meager} if its complement is meager. By Baire's
category theorem a set is co-meager iff it contains a dense $G_\delta$ set. We
say that a \emph{generic} element $x \in X$ has property $\mathcal{P}$ if
$\{x \in X : x \textrm{ has property } \mathcal{P} \}$ is co-meager. The set
$A\subset X$ has the \emph{Baire property} if $A=U\Delta M$ where $U$ is open
and $M$ is meager.
If $A\cap U$ contains a set which is of second category and has the Baire property for all
non-empty open sets $U\subset X$, then $A$ is co-meager.
A separable, complete metric space is called a \emph{Polish space}.
If $X$ is a Polish space and $A\subset X$, then $A$ is \emph{analytic}
if it is a continuous image of a Polish space.
Let $\sigma(\mathbf{A})$ denote the $\sigma$-algebra generated by the analytic sets.
Clearly $\sigma(\mathbf{A})$ is closed under taking continuous images.
Every analytic set has the Baire property, so sets in $\sigma(\mathbf{A})$ have the Baire property, too.
For more information see \cite{Ke}.

If $E\subset X\times Y$ then for $x\in X$ let $E_x=\{y\in Y: (x,y)\in E\}$, and
for $y\in Y$ let $E^y=\{x\in X: (x,y)\in E\}$. Let $\mathbf{0}$ denote the origin of $\RR^n$.

If $X$ is a metric space and $A,B$ are disjoint subsets of $X$
then we say that $P\subset X$ is a \emph{partition between
$A$ and $B$} if there are open sets $U$, $V$ such that $A\subset U$,
$B\subset V$, $U\cap V=\emptyset$, and $P=X\setminus (U\cup V)$.

Topological dimension admits $G_{\delta}$ hulls, see \cite[Theorem~1.5.11]{E}.

\begin{theorem}[Enlargement theorem] Let $X$ be a metric space with a separable subspace $Y\subset X$.
There is a $G_{\delta}$ set $G\subset X$ such that $Y\subset G$ and $\dim_{T} G=\dim_{T} Y$.
\end{theorem}

For the following theorem see \cite[Theorem~1.5.3]{E}.

\begin{theorem}[Countable stability for $F_{\sigma}$ sets] Let $X$ be a separable metric space and let $X_i\subset X$ be
$F_{\sigma}$ subsets such that $X=\bigcup_{i=1}^{\infty} X_i$. Then
\[\dim_T X=\sup_{i\geq 1} \dim_T X_i.\]
\end{theorem}

The next result is \cite[Theorem~1.5.10]{E}.
\begin{theorem}[Addition theorem] \label{t:add}
If $X,Y$ are separable subspaces of a metric space then
\[\dim_T(X\cup Y)\leq \dim_T X+\dim_T Y+1.\]
\end{theorem}

See \cite[Theorem~1.5.13]{E} for the following.
\begin{theorem}[Separation theorem]
Let $X$ be a metric space with a separable subspace $Y\subset X$ and let $n\in \NN^+$.
If $\dim_T Y\leq n$, then for every pair $(A,B)$ of disjoint closed subsets of $X$ there is a partition
$P$ between $A$ and $B$ such that $\dim_T (P\cap Y)\leq n-1$.
\end{theorem}

The following theorem is \cite[Theorem~1.5.8]{E}.

\begin{theorem}[Decomposition theorem] \label{t:dec}
Let $X$ be a separable metric space and let $n\in \mathbb{N}^+$. Then the following statements are equivalent:
\begin{enumerate}[(i)]
\item $\dim_T X\leq n$;
\item $X=Z_1\cup\dots \cup Z_{n+1}$ such that $\dim_T Z_i\leq 0$ for all $i\in \{1,\dots,n+1\}$.
\end{enumerate}
\end{theorem}

For the next theorem see \cite[Theorem~1.7.9]{E}.

\begin{theorem}[Theorem on partitions] Let $X$ be a separable metric space and let $n\in \NN^+$.
Then the following statements are equivalent:
\begin{enumerate}[(i)]
\item $\dim_T X\leq n$;
\item for every sequence $(A_1,B_1),\dots, (A_{n+1},B_{n+1})$ of $n+1$ pairs of
disjoint closed subsets of $X$ there are partitions $P_i$ between $A_i$ and $B_i$ such that
$\bigcap_{i=1}^{n+1} P_i=\emptyset$.
\end{enumerate}
\end{theorem}

Let $X$ be a metric space and let $A\subset X$.
We use the convention $0^0=1$, and $\inf \emptyset=\infty$.
We say that $h\colon [0,\infty)\to [0,\infty)$ is a
\emph{gauge function} if it is non-decreasing.
We define the \emph{$h$-Hausdorff measure} of $A$ as
\begin{align*} \mathcal{H}^{h}(A)&=\lim_{\delta \to 0+} \iH^h_{\delta}(A), \textrm{ where} \\
\iH^h_{\delta}(A)&=\inf \left\{ \sum_{i=1}^\infty h(\diam A_{i}) : A\subset \bigcup_{i=1}^{\infty} A_{i},~ \forall i \, \diam A_i \leq \delta \right\}.
\end{align*}
If $h(x)=x^s$ for some $s\geq 0$ then we use the notation $\iH^h=\iH^s$ and $\iH^h_{\delta}=\iH^s_{\delta}$.
The \emph{Hausdorff dimension} of a non-empty $A$ is defined as
\[\dim_{H} A= \inf\{s \ge 0: \mathcal{H}^s(A) =0\}.\]
The regularity of $\mathcal{H}_{\delta}^{s}$ easily implies that
every set is contained in a $G_\delta$ set of the same Hausdorff dimension, and it is easy to see that countable stability holds.

For every $\delta>0$, a \emph{$\delta$-packing} of $A$ is a countable collection of disjoint balls $\{B(x_i,r_i)\}_{i\geq 1}$
with centers $x_i\in A$ and radii $0\leq r_i \leq \delta$. For a gauge function $h$ let $P_0^h(\emptyset)=0$ and define the \emph{$h$-packing number} of a non-empty $A$ as
\begin{align*} P_{0}^h(A)&=\lim_{\delta \to 0+} P^h_{\delta}(A), \textrm{ where} \\
P^h_{\delta}(A)&=\sup\left\{ \sum_{i=1}^{\infty} h(r_i): \{B(x_i,r_i)\}_{i\geq 1} \textrm{ is a $\delta$-packing of } A\right\}.
\end{align*}
As the countable subadditivity does not hold
for $P_{0}^h$, we consider the following modification.
The \emph{$h$-packing measure} of $A$ is defined as
\[ \iP^h(A)=\inf \left\{\sum_{i=1}^{\infty} P_{0}^h(A_i): A=\bigcup_{i=1}^{\infty} A_i \right\}. \]
If $h(x)=x^s$ for some $s\geq 0$ then we use the notation $\iP^s=\iP^h$.
The \emph{packing dimension} of a non-empty $A$ is defined as
\[ \dim_P A=\sup\{ s\geq 0: \iP^s(A)=\infty\}. \]

Countable stability holds for packing dimension, but it does not admit $G_{\delta}$ hulls.

\begin{fact} \label{f:close} Let $h$ be a left-continuous gauge function and let
$X$ be a metric space. Then $P_{0}^h(A)=P_0^h(\cl A)$ for all $A\subset X$.
\end{fact}

For the following see \cite[Lemma~4]{H}.

\begin{lemma} \label{l:pack1}
Let $K$ be a compact metric space and let $h$ be a left-continuous
gauge function. If $P_{0}^h(U)=\infty$ for all non-empty open sets
$U\subset K$, then $\iP^h|_{K}$ is not $\sigma$-finite.
\end{lemma}

The following lemma is \cite[Lemma~2.8.1~(ii)]{BP}. 

\begin{lemma}\label{l:pack} If $X$ is a separable metric space satisfying $\dim_P X>s$,
then there is a closed set $C\subset X$ such that $\dim_P (C\cap U)>s$ for all open sets $U$ with $C\cap U\neq \emptyset$.
\end{lemma}

For the following lemma see \cite{Ho}.

\begin{lemma} \label{l:dimp} Let $X$ be a non-empty metric space and let $n\in \NN^+$. Then
\[ \dim_P (X\times [0,1]^n)=\dim_P X+n. \]
\end{lemma}

For more information on Hausdorff and packing dimensions see \cite{BP}, \cite{F}, and \cite{Ma}.

The following lemma is standard, see e.g.\ \cite[Lemma~3.8]{BFFH}.

\begin{lemma} \label{l:cat}
Let $X,Y$ be complete metric spaces and let $R\colon X\to Y$ be a
continuous open map. If $B\subset Y$ is co-meager, then
$R^{-1}(B)\subset X$ is co-meager, too.
\end{lemma}

As Tietze's extension theorem holds in $\RR^n$, Lemma~\ref{l:cat} implies the following.

\begin{corollary} \label{c:R}
Let $K_1\subset K_2$ be compact metric spaces, and let $n\in \NN^+$. Define
\[ R \colon C(K_2,\RR^n)\rightarrow C(K_1,\RR^n), \quad R(f)=f|_{K_1}.\]
If $\mathcal{F}_1\subset C(K_1,\RR^n)$ is co-meager,
then $R^{-1}(\mathcal{F}_1)\subset C(K_2,\RR^n)$ is co-meager, too.
\end{corollary}

For the following theorem see \cite[Proposition~3.2]{K}.

\begin{theorem}[Kato]\label{t:kato} Let $n\in \NN^+$ and let $K$ be a compact metric space. Then for a generic $f\in C(K,\RR^n)$ we have
$\dim_T f(K)\leq \dim_T K$.
\end{theorem}

\begin{lemma} \label{l:int} Let $n\in \NN^+$ and let $K$ be a compact metric space with $\dim_T K\geq n$.
Let $V$ be the maximal open set $V\subset K$ with $\dim_T V<n$ and let $C=K\setminus V$. Then
\begin{enumerate}[(i)]
\item \label{int1} $\dim_T (C\cap U)\geq n$ for all open sets $U\subset K$ intersecting $C$,
\item \label{int2} $\inter f(C)=\inter f(K)$ for a generic $f\in C(K,\RR^n)$.
\end{enumerate}
\end{lemma}

\begin{proof}
The maximal open set $V\subset K$ with $\dim_T V<n$ can be defined as
\[ V=\bigcup \{W \subset K: W \textrm{ is open and } \dim_T W<n\}.\]
Indeed, the countable stability of topological dimension for
$F_{\sigma}$ sets and the Lindel\"of property of $V$ imply that $\dim_T V<n$.

First we prove \eqref{int1}. Assume to the contrary that there is an open set $U\subset K$ intersecting $C$ such that
$\dim_T (C\cap U)<n$. Clearly $\dim_T (U\setminus C)\leq \dim_T V<n$, so the countable stability of topological dimension for
$F_{\sigma}$ sets implies that $\dim_T U<n$. Then $U\subset V$ by definition, thus $C\cap U=\emptyset$.
This is a contradiction, so \eqref{int1} holds.

Now we show \eqref{int2}. Let $K_i$ be compact sets such that $V=\bigcup_{i=1}^{\infty} K_i$.
Theorem~\ref{t:kato} and Corollary~\ref{c:R} imply that for all $i\in \NN^+$
a generic $f\in C(K,\RR^n)$ satisfies $\dim_T f(K_i)\leq n-1$.
As a countable intersection of co-meager sets is co-meager and topological dimension is
countably stable for closed sets, the set
\[ \iF=\{f\in C(K,\RR^n): \dim_T f(V)\leq n-1\}\]
is co-meager in $C(K,\RR^n)$. Assume $f\in \iF$ and $U=\inter f(K)$, we need to prove that $U\subset f(C)$.
Since $\dim_T f(V)<n$, the set $U\setminus f(V)$ is dense in $U$, and clearly
$U\setminus f(V) \subset f(C)$. As $f(C)$ is compact, we have $U\subset f(C)$. This proves \eqref{int2}.
\end{proof}

\section{Some basic properties of $d_{*}^n$} \label{s:basic}

Although our main examples for $\dim_{*}$ will
be $\dim_T$, $\dim_H$, and $\dim_P$,
we might consider more general dimensions as well.

\begin{definition}
We say that $\dim_{*}$ is a \emph{notion of dimension}, if
\begin{enumerate}[(i)]
\item $\dim_{*} \emptyset=-1$ and $\dim_{*}\{x\}=0$ for each singleton $x$,
\item $\dim_{*} X\leq \dim_{*} Y$ for all separable metric spaces $X\subset Y$.
\end{enumerate}
\end{definition}

We can extend Definition~\ref{d:dn} as follows.

\begin{definition}  Let $n\in \NN^+$, let $X$ be a separable metric space, and
let $\dim_{*}$ be a notion of dimension. Then
\[ d^n_{*}(X)=\inf \{\dim_{*} (X\setminus F): F\subset X \textrm{ is an } F_{\sigma}\textrm{ set with } \dim_T F\leq n-1\}. \]
\end{definition}

\begin{fact} \label{f:inf} The infimum is attained in the above definition, that is,
\[ d^n_{*}(X)=\min \{\dim_{*} (X\setminus F): F\subset X \textrm{ is an } F_{\sigma}\textrm{ set with } \dim_T F\leq n-1\}.\]
\end{fact}

\begin{proof} Assume that $F_i\subset X$ are $F_{\sigma}$-sets such that $\dim_T F_i\leq n-1$ for all $i\in \NN^+$ and
$\dim_{*} (X\setminus F_i)\to d^n_{*} (X)$ as $i\to \infty$. Let $F=\bigcup_{i=1}^{\infty} F_i$, then $F$ is $F_{\sigma}$, too.
As $X$ is separable, the countable stability of topological dimension for $F_{\sigma}$
sets implies that $\dim_T F\leq n-1$.
The monotonicity of $\dim_{*}$ yields that
\[ \dim_{*} (X\setminus F)\leq \inf_{i\geq 1} \dim_{*}(X\setminus F_i)=d_{*}^n(X), \]
which completes the proof.
\end{proof}

\begin{fact} \label{f:equiv} Let $X$ be a separable metric space, let $\dim_{*}$ be a notion of dimension, and let $n\in \NN^+$. Then
\[ d^{n}_{*}(X)=-1 \Longleftrightarrow \dim_T X\leq n-1 \quad \textrm{and} \quad d^{n}_{*}(X)\geq 0 \Longleftrightarrow \dim_T X\geq n.\]
\end{fact}

The following fact easily follows from the definition of $d^n_{*}$.

\begin{fact}[Monotonicity] Let $\dim_{*}$ be a notion of dimension and let $X\subset Y$ be separable metric spaces.
Then $d^n_{*}(X)\leq d^n_{*}(Y)$ for all $n\in \NN^+$.
\end{fact}
For the following theorem see \cite[Theorem~VII~2.]{HW} and \cite{F}.

\begin{theorem} \label{t:<} For every separable metric space $X$ we have
\[ \dim_T X\leq \dim_H X\leq \dim_P X. \]
\end{theorem}

The above theorem immediately implies the following.

\begin{fact} For every separable metric space $X$ and $n\in \NN^+$ we have
\[ d^n_{T}(X)\leq d^n_{H}(X)\leq d^n_{P}(X).\]
\end{fact}

We are able to explicitly calculate $d^n_T(X)$.

\begin{theorem} \label{t:dtn} If $n\in \NN^+$ and $X$ is a separable metric space with $\dim_T X \geq n$ then
\[ d_{T}^n (X)=\dim_T X-n. \]
\end{theorem}

\begin{proof} By Fact~\ref{f:inf} there exists an $F_{\sigma}$ set $F\subset X$ such that $\dim_T F\leq n-1$ and $d^n_{T}(X)=\dim_T (X\setminus F)$.
By the addition theorem we have
\[ \dim_T X\leq \dim_T (X\setminus F)+\dim_T F+1\leq d^{n}_T(X)+n. \]
Hence $d^{n}_T(X)\geq \dim_T X-n$.

For the other direction let $\dim_T X=d$. The decomposition theorem implies that
$X$ can be represented as a union $X=Z_1\cup \dots \cup Z_{d+1}$, where $\dim_T Z_i\leq 0$ for all $1\leq i\leq d+1$.
As topological dimension admits $G_{\delta}$ hulls,
there exists a $G_{\delta}$ set $G\subset X$ such that $Z_{n+1}\cup \dots \cup Z_{d+1}\subset G$ and
$\dim_T G=\dim_T(Z_{n+1}\cup \dots \cup Z_{d+1})$. Then $F=X\setminus G$ is an $F_{\sigma}$ set. Monotonicity and the addition theorem yield that
\[ \dim_T F\leq \dim_T(Z_1\cup \dots \cup Z_n)\leq n-1. \]
Similarly, we have
\[ \dim_T(X\setminus F)=\dim_T G=\dim_T(Z_{n+1}\cup \dots \cup Z_{d+1})\leq d-n=\dim_T X-n. \]
Therefore $F$ witnesses that $d^{n}_T(X)\leq \dim_T X-n$.
\end{proof}

Theorem~\ref{t:equ} and \cite[Theorem~3.5]{B} imply an upper estimate for $d^n_H(X)$.

\begin{theorem} \label{t:dHn} If $n\in \NN^+$ and $X$ is a separable metric space with $\dim_T X\geq n$ then
\[ d^n_H(X)\leq \dim_H X-n. \]
\end{theorem}

By the product of two metric spaces $(X,d_X)$ and $(Y,d_Y)$ we mean the $\ell^2$-product
\[ d_{X\times Y}((x_{1},y_{1}),(x_{2},y_{2}))= \sqrt{d^{2}_{X}(x_1,x_2)+d^{2}_{Y}(y_1,y_2)}. \]
It turns out that the upper estimate of
Theorem~\ref{t:dHn} is sharp for product sets of the form
$X\times [0,1]^n$. For the following lemma see \cite[Lemma~5.2]{B}.

\begin{lemma} \label{l:prodH} Let $X$ be a non-empty separable metric space and let $n\in \mathbb{N}^+$. Then
\[ d^n_{H}(X \times [0,1]^n)=\dim_{H} X. \]
\end{lemma}

For $d^n_P(X)$ there is no non-trivial upper bound, $d^n_P(X)=\dim_P X$ is possible.
We show that this is the case for product sets of the form $X\times [0,1]^n$.

\begin{lemma} \label{l:prodP} Let $X$ be a non-empty Polish space and let $n\in \mathbb{N}^+$. Then
\[ d^n_{P}(X \times [0,1]^n)=\dim_P(X\times [0,1]^n)=\dim_{P} X+n. \]
\end{lemma}

\begin{proof} Lemma~\ref{l:dimp} implies that $d^n_{P}(X \times [0,1]^n)\leq \dim_P(X \times [0,1]^n)=\dim_P X+n$.

For the other direction let $Z=X\times [0,1]^n$ and let $d<\dim_P X$ be fixed,
it is enough to prove that $d^n_{P}(Z)\geq d+n$. Let $F\subset Z$ be an $F_{\sigma}$ set with
$\dim_T F\leq n-1$, it is sufficient to show that $\iP^{d+n}(Z\setminus F)=\infty$. Let $\{A_i\}_{i\geq 1}$ be an arbitrary cover of
$Z\setminus F$, it is enough to prove that \begin{equation} \label{eq:p0a} P_{0}^{d+n}(A_i)=\infty \quad \textrm{for some } i.
\end{equation} 
By Lemma~\ref{l:pack} there is a closed set $Y\subset X$ such that every non-empty relatively open set $U$ in $Y$ satisfies
$\dim_P U>d$. Now we prove that for every non-empty relatively open set $V\subset Y\times [0,1]^n$ we have 
\begin{equation} \label{eq:p0v}  P_{0}^{d+n}(V)=\infty.
\end{equation} 
Indeed, there is a relatively open set $U$ in $Y$ and a cube $Q\subset [0,1]^n$
such that $U\times Q\subset V$, so applying Lemma~\ref{l:dimp} we obtain that
\[ \dim_P V\geq \dim_P(U\times Q)=\dim_P U+n>d+n. \]
Therefore $\iP^{d+n}(V)=\infty$, so $P_{0}^{d+n}(V)=\infty$. Hence \eqref{eq:p0v} holds.

Clearly $\dim_T F\leq n-1$ yields that $F$ is meager in $\{y\}\times [0,1]^n$ for all $y\in Y$, so $F$ is meager
in $Y\times [0,1]^n$ by the Kuratowski--Ulam Theorem \cite[Theorem~8.41]{Ke}. Since $Y\times [0,1]^n$ is complete,
Baire's category theorem implies that for some $i$ the set
$A_i$ is somewhere dense in $Y\times [0,1]^n$, so there is a non-empty relatively open set
$V\subset Y\times [0,1]^n$ such that $V\subset \cl A_i$. Thus Fact~\ref{f:close} and \eqref{eq:p0v} yield that
\[ P_{0}^{d+n}(A_i)=P_{0}^{d+n}(\cl A_i)\geq P_{0}^{d+n}(V)=\infty. \]
Hence \eqref{eq:p0a} holds, and the proof is complete.
\end{proof}

\begin{theorem} \label{t:stable} Let $n\in \NN^+$ and let $\dim_{*}$ be a notion of dimension. Assume that one of the following holds:
\begin{enumerate}[(i)]
\item $\dim_{*}$ is countably stable,
\item $\dim_{*}$ is countably stable for closed sets and admits $G_{\delta}$ hulls.
\end{enumerate}
Then $d^n_{*}$ is countably stable for closed sets.
\end{theorem}

\begin{proof}
Let $X$ be a separable metric space and assume that
$X_i\subset X$ are closed such that $X=\bigcup_{i=1}^{\infty} X_i$, we have to prove that $d^n_{*}(X)=\sup_{i} d^n_{*} (X_{i})$.

Monotonicity clearly implies $d^n_{*}(X) \geq \sup_{i} d^n_{*} (X_{i})$, so it is enough to prove the other direction.
By Fact~\ref{f:inf} for each $i\in \NN^+$ there exists an $F_{\sigma}$ set $F_i\subset X_i$ such that $\dim_T F_i\leq n-1$ and
$\dim_{*}(X_i\setminus F_i)=d^n_{*}(X_i)$. Let $F=\bigcup_{i=1}^{\infty} F_i$.
As $X_i$ are closed, $F$ is $F_{\sigma}$ in $X$. As $F_i$ are $F_{\sigma}$ in $F$ as well,
the countable stability of topological dimension for $F_{\sigma}$-sets implies that $\dim_T F\leq n-1$.

First assume that $\dim_{*}$ is countably stable, then
\[ d^n_{*}(X)\leq \dim_{*}(X\setminus F)=\sup_{i\geq 1} \dim_{*}(X_i\setminus F)=\sup_{i\geq 1} d^{n}_{*} (X_i)\]
completes the proof. 

Now assume that $\dim_{*}$ is countably stable for closed sets and admits $G_{\delta}$ hulls.
As the topological dimension admits $G_{\delta}$ hulls,
there exists a $G_{\delta}$ set $G\subset X$
such that $F\subset G$ and $\dim_{T} G\leq n-1$. 
Clearly $\dim_{*}$ is countably stable for $F_{\sigma}$ sets, too. This and monotonicity yield that
\begin{equation} \label{eq:Gd1} \dim_{*} (X\setminus G)=\sup_{i\geq 1} \dim_{*} (X_i\setminus G)
\leq \sup_{i\geq 1} \dim_{*} (X_i\setminus F_i)=\sup_{i\geq 1} d^n_{*} (X_{i}).
\end{equation}
Using countable stability for $F_{\sigma}$ sets again implies that there is a $G_{\delta}$ set $H\subset X$
such that $X\setminus G\subset H$ and $\dim_{*} H=\dim_{*} (X\setminus G)$. Then
$X\setminus H$ is $F_{\sigma}$ and we have $\dim_T (X\setminus H)\leq \dim_T G\leq n-1$.
Thus the definition of $d_{*}^n(X)$ and \eqref{eq:Gd1} imply that
\begin{equation*} d_{*}^n(X)\leq \dim_{*}H=\dim_{*} (X\setminus G)\leq \sup_{i\geq 1} d^n_{*} (X_{i}).
\end{equation*}
The proof is complete.
\end{proof}

\begin{corollary} \label{c:stable} Let $n\in \NN^+$ and let $\dim_{*}$ be one of $\dim_T$, $\dim_H$, or $\dim_P$. Then
$d^{n}_{*}$ is countably stable for $F_{\sigma}$ sets.
\end{corollary}

\section{The Main Theorem} \label{s:main}

The goal of this section is to describe the dimensions of the fibers of a generic $f\in C(K,\RR^n)$ by proving our Main Theorem.
Note that the case $\dim_T K <n$ is completed by the following theorem of Hurewicz~\cite[page~124]{Ku}.

\begin{theorem}[Hurewicz] \label{t:Hurewicz}
Let $n\in \NN^+$ and let $K$ be a compact metric space with $\dim_T K<n$. Then $\#f^{-1}(y)\leq n$
for a generic $f\in C(K,\RR^n)$ for all $y\in \mathbb{R}^n$.
\end{theorem}

Consider the following technical version of the Main Theorem.

\begin{theorem}[Main Theorem] \label{mainthm} Let $n\in \NN^+$ and let $K$ be a compact metric space with $\dim_T K\geq n$.
Let $\dim_{*}$ be one of $\dim_T$, $\dim_H$, or $\dim_P$. Then
there exists a $G_{\delta}$ set $G\subset K$ with $\dim_{*} G=d^n_{*}(K)$ such that
for a generic $f\in C(K,\RR^n)$ we have
\begin{enumerate}[(i)]
\item \label{a} $\#(f^{-1}(y)\setminus G)\leq n$ for all $y\in \RR^n$, thus $\dim_{*} f^{-1} (y)\leq d^n_{*}(K)$ for all $y\in \RR^n$,
\item \label{b} for every $d<d^n_{*}(K)$ there exists a non-empty open set $U_{f,d}\subset \RR^n$ such that $\dim_{*} f^{-1}(y)\geq d$ for all $y\in U_{f,d}$.
\end{enumerate}
\end{theorem}

\begin{corollary} \label{c:main1}
Let $K$ be a compact metric space with $\dim_T K\geq n$. Let $\dim_{*}$ be one of $\dim_T$, $\dim_H$, or $\dim_P$.
Then for a generic $f\in C(K,\RR^n)$ we have
\[ \sup\{\dim_{*}f^{-1}(y): y\in \mathbb{R}^n\}=d^n_{*}(K).\]
\end{corollary}

Theorem~\ref{t:Hurewicz} yields the following corollary, which states that a generic continuous map $f\in C(K,\RR^n)$ is almost injective on
any given $F_{\sigma}$ set of topological dimension smaller than $n$.

\begin{corollary} \label{c:Hur} Let $K$ be a compact metric space and let $n\in \NN^+$. Assume that $F\subset K$ is an $F_{\sigma}$ set such that
$\dim_T F<n$. Then $\#(f^{-1}(y)\cap F)\leq n$ for a generic $f\in C(K,\RR^n)$ for all $y\in \mathbb{R}^n$.
\end{corollary}

\begin{proof}
Choose compact sets $K_i$ such that $F=\bigcup_{i=1}^{\infty} K_i$ and $K_i\subset K_{i+1}$ for all $i$. For all $i\in \mathbb{N}^+$ let
\[ \mathcal{F}_i=\{f\in C(K_i,\RR^n): \#f^{-1}(y)\leq n \textrm{ for all } y\in \mathbb{R}^n  \}, \]
and define
\[ R_i\colon C(K,\RR^n)\to C(K_i,\RR^n),\quad R_i(f)=f|_{K_i}.\]
Finally, let us define
\[ \mathcal{F}=\bigcap_{i=1}^{\infty} R_i^{-1}(\mathcal{F}_i)\subset C(K,\RR^n).\]
As $\dim_T K_i\leq \dim_T F\leq n-1$, Theorem~\ref{t:Hurewicz} yields that the sets
$\mathcal{F}_i\subset C(K_i,\RR^n)$ are co-meager. Corollary~\ref{c:R} implies that $R_i^{-1}(\mathcal{F}_i)\subset C(K,\RR^n)$
are co-meager, too. As a countable intersection of co-meager sets $\iF\subset C(K,\RR^n)$ is also co-meager.
Fix $f\in \mathcal{F}$, we have $\#(f^{-1}(y)\cap K_i)\leq n$ for all $y\in \mathbb{R}^n$ and $i\in \mathbb{N}^+$.
Thus $\bigcup_{i=1}^{\infty} K_i=F$ and $K_i\subset K_{i+1}$ imply that $\# (f^{-1}(y)\cap F)\leq n$ for all
$f\in \mathcal{F}$ and $y\in \mathbb{R}^n$.
\end{proof}

Therefore we may remove an $F_{\sigma}$ set $F$ of topological dimension smaller than $n$ from $K$, and
the dimension of $K\setminus F$ estimates the dimension of the fibers of a generic continuous map from above. This yields
the first half of the Main Theorem.

\begin{proof}[Proof of Main Theorem~\eqref{a}] By Fact~\ref{f:inf} there is an $F_{\sigma}$ set $F\subset K$ such that $\dim_T F\leq n-1$ and $d^n_{*}(K)=\dim_{*}(K\setminus F)$. Let $G=K\setminus F$, then $\dim_{*} G=d^n_{*}(K)$ and Corollary~\ref{c:Hur} implies that for a generic $f\in C(K,\RR^n)$ for all $y\in \RR^n$ we have $\#(f^{-1}(y)\setminus G)=\#(f^{-1}(y)\cap F)\leq n$. As $\dim_T K\geq n$ yields $G\neq \emptyset$, we have $\dim_{*} f^{-1}(y)\leq \dim_{*} G=d^{n}_{*}(K)$ for a generic $f\in C(K,\RR^n)$ for all $y\in \RR^n$.
\end{proof}

The second part of the Main Theorem is much deeper, it states that the easy upper bound of the first part is always sharp.
In case of Hausdorff dimension this is the main result of \cite{B}, see Theorem~\ref{t:B}.
In the following two subsections we prove the Main Theorem in case of topological and packing dimensions.

\subsection{Topological dimension} \label{ss:top}

The goal of this subsection is to prove Theorem~\ref{t:top}, which is the second part of the Main Theorem for topological dimension.
The proof is based on the proof of \cite[Theorem~6.12]{B}, but it is much shorter. This is because we can use the
powerful theory of topological dimension, which in particular implied Theorem~\ref{t:dtn}.

\begin{theorem} \label{t:top} Let $n\in \NN^+$ and let $K$ be a compact metric space with $\dim_T K\geq n$. Then for a generic
$f\in C(K,\RR^n)$ for all $d<d_T^{n}(K)$ there is a non-empty open set $U_{f,d}\subset \RR^n$ such that for all $y\in U_{f,d}$ we have
$\dim_T f^{-1}(y)\geq d$.
\end{theorem}

First we need some preparation. Assume that $K$ is fixed with $\dim_T K\geq n$.

\begin{definition} We say that $f\in C(K,\RR^n)$ is \emph{$d$-level narrow} if there is a dense set $S_f\subset \RR^n$ such that $\dim_{T} f^{-1}(y)\leq d$ for all $y\in S_f$. Let $\iN^n(d)$ be the set of $d$-level narrow maps. Let
\begin{align*} N_n=\{&d: \iN^n(d) \textrm{ is somewhere dense in } C(K,\RR^n)\}, \\
D_n=\{&d: \exists \textrm{ an } F_{\sigma} \textrm{ set } F\subset K \textrm{ with } \dim_T F\leq n-1 \textrm{ and } \dim_{T} (K\setminus F)\leq d\}.
\end{align*}
We adopt the convention that $\infty\in N_n, D_n$.
\end{definition}

\begin{theorem} \label{t:3eq} Let $n\in \NN^+$ and let $K$ be a compact metric space with $\dim_T K\geq n$. Then
\[d^n_{T}(K)=\min N_n.\]
\end{theorem}

The following two lemmas and Fact~\ref{f:inf} imply Theorem~\ref{t:3eq}.

\begin{lemma} $D_n\subset N_n$.
\end{lemma}

\begin{proof}
Assume that $d\in D_n$ and $d<\infty$. The first part of the Main Theorem yields that for a generic $f\in C(K,\RR^n)$
we have $\dim_T f^{-1}(y)\leq d$ for all $y\in \RR^n$. Therefore
$\iN^n(d)$ is co-meager, thus everywhere dense in $C(K,\RR^n)$. Hence $d\in N_n$.
\end{proof}

The proof of the following lemma is based on the proof of \cite[Lemma~6.11]{B}.

\begin{lemma} $N_n\subset D_n$.
\end{lemma}

\begin{proof} Assume that $d\in N_n$ is a natural number, by Fact~\ref{f:inf} it is enough to prove
that $d^n_T(K)\leq d$.
As $d^{n}_T(K)=\dim_T K-n$ by Theorem~\ref{t:dtn}, it is enough to show that $\dim_T K\leq n+d$.
Fix $f\in C(K,\RR^n)$ and $\varepsilon>0$
such that $\mathcal{N}^n(d)$ is dense in $B(f,\varepsilon)$. The uniform continuity of $f$ implies that
there is a $\delta>0$ such that if $A\subset K$ with $\diam A\leq \delta$ then $\diam f(A)<\varepsilon/n$. Decompose $K$ into finitely many compact sets with diameter less than or equal to $\delta$. As topological dimension is (countably) stable for closed sets, there is a compact set $C\subset K$ such that $\diam C\leq \delta$ and $\dim_T C=\dim_T K$. Thus we may assume that $\diam K\leq \delta$, so $\diam f(K)<\varepsilon/n$.

Let $(A_1,B_1),\dots,(A_{n+d+1},B_{n+d+1})$ be arbitrary pairs of disjoint closed subsets of
$K$, then by the theorem on partitions
it is enough to prove that there exist partitions $P_i\subset K$ between $A_i$ and $B_i$ such that
$\bigcap_{i=1}^{n+d+1} P_i=\emptyset$. First we show that it is enough to find partitions
$P_i\subset K$ between $A_i$ and $B_i$ for $1\leq i\leq n$ such that
\begin{equation} \label{eq:par} \dim_T \left(\bigcap_{i=1}^{n} P_i\right)\leq d.
\end{equation}
Indeed, assuming \eqref{eq:par} and applying the separation theorem $(d+1)$-times
yields that for all $1\leq j\leq d+1$ there are partitions
$P_{n+j}$ between $A_{n+j}$ and $B_{n+j}$ with
\[ \dim_{T} \left(\bigcap_{i=1}^{n+j} P_i\right)\leq d-j. \] 
Therefore $\dim_T(\bigcap_{i=1}^{n+d+1} P_i)=-1$, so $\bigcap_{i=1}^{n+d+1} P_i=\emptyset$.

Finally, we prove \eqref{eq:par}.
Let $f_1, \dots, f_n \in C(K,\RR)$ be such that $f =(f_1, \dots, f_n)$ and observe
that we may construct for all $i\in \{1,\dots,n\}$ functions $g_i\in C(K,\RR)$ such that
\begin{enumerate}[(i)]
\item \label{eq:C1} $\max g_i(A_i)< \min g_i(B_i)$;
\item \label{eq:C2} $g_i \in B(f_i, \varepsilon/n)$;
\item \label{eq:C3} The map $g= (g_1, \dots, g_n) \in C(K,\RR^n)$ satisfies $g \in \mathcal{N}^n(d)$.
\end{enumerate}
Indeed, as $\diam f_i(K)\leq \diam f(K)<\varepsilon/n$, for every $i\in \{1,\dots,n\}$ we can define $g_i$ first on
$A_i\cup B_i$ and then we can extend it to $K$ by Tietze's extension theorem such that
\eqref{eq:C1} and \eqref{eq:C2} hold. Property~\eqref{eq:C2} implies that $g=(g_1,\dots,g_n)\in B(f,\varepsilon)$.
As $g\in B(f,\varepsilon)$ and $\mathcal{N}^n(d)$ is dense in $B(f,\varepsilon)$, we may assume that $g\in \mathcal{N}^n(d)$, so \eqref{eq:C3}
holds.

As $g\in \mathcal{N}^n(d)$,
there is a dense set $S_g\subset \mathbb{R}^n$ such that $\dim_T g^{-1}(s)\leq d$ for all
$s\in S_g$. We can choose $s=(s_1,\dots,s_n)\in S_g$ such that for every $i\in \{1,\dots,n\}$
its $i$th coordinate $s_i$ satisfies
$\max g_i(A_i)<s_i < \min g_i(B_i)$. For all $i\in \{1,\dots,n\}$ let
\[ S_i=\{(y_1,\dots,y_n) \in g(K): y_i=s_i\}. \]
Then \eqref{eq:C1} implies that $S_i$ is a partition between $g(A_i)$ and $g(B_i)$ in $g(K)$ for every
$i\in \{1,\dots,n\}$. For all $i$ define $P_i=g^{-1}(S_i)$.
Then $P_i$ is a partition between $A_i$ and $B_i$ such that
\[ \bigcap_{i=1}^{n}P_i=\bigcap_{i=1}^{n}g^{-1}(S_i)=
g^{-1}\left( \bigcap_{i=1}^{n} S_i \right)=g^{-1}(s). \]
Therefore $s\in S_g$ implies that
\[ \dim_T \left(\bigcap_{i=1}^{n}P_i\right)\leq \dim_T g^{-1}(s)\leq d. \]
Thus \eqref{eq:par} holds, and the proof is complete.
\end{proof}

Now we are ready to prove Theorem~\ref{t:top}.

\begin{proof}[Proof of Theorem~\ref{t:top}] By Theorem~\ref{t:3eq} we have
$d^n_T(K)=\min N_n$. Let us choose a sequence $d_k \nearrow d^n_{T}(K)$, then $\mathcal{N}^{n}(d_k)$ is nowhere
dense by the definition of $N_n$. Thus for every $f\in C(K,\RR^n)\setminus
\mathcal{N}^{n}(d_k)$ there exists a non-empty open set
$U_{f,d_k}\subset \mathbb{R}^n$ such that $\dim_{T}f^{-1}(y)\geq d_k$ for every $y\in U_{f,d_k}$. But then the theorem holds
for every $f\in C(K,\RR^n)\setminus \left(\bigcup_{k=1}^{\infty} \mathcal{N}^{n}(d_k)\right)$, and this latter set is
co-meager.
\end{proof}

Theorem~\ref{t:top} clearly implies the following known result, which will be useful later.

\begin{theorem} \label{t:int} Let $n\in \NN^+$ and let $K$ be a compact metric space with $\dim_T K\geq n$.
Then $\inter f(K)\neq \emptyset$ for a generic $f\in C(K,\RR^n)$.
\end{theorem}

The origins of the above theorem date back to Hurewicz and Wallman~\cite{HW}, and Alexandroff~\cite{A}.
See the remark after \cite[Theorem~1.7]{BFFH} for more explanation.

\subsection{Packing dimension} \label{ss:packing}

The main goal of this subsection is to prove the following theorem, which is the second part of the Main Theorem for packing dimension. Topological and Hausdorff dimensions admit $G_{\delta}$ hulls, which was crucial in the proof of the Main Theorem for these dimensions. Since packing dimension does not admit $G_{\delta}$ hulls, we need to come up with different arguments.

\begin{theorem} \label{t:mainP} Let $n\in \NN^+$ and let $K$ be a compact metric space with $\dim_T K\geq n$. For a generic $f\in C(K,\RR^n)$ for every $d<d^n_{P}(K)$ there exists a non-empty open set $U_{f,d}\subset \RR^n$ such that for all $y\in U_{f,d}$ we have $\dim_{P} f^{-1}(y)\geq d$.
\end{theorem}

Actually, we prove a stronger statement by using the concept of generalized packing measure.

\begin{theorem} \label{t:gauge} Let $n\in \NN^+$ and let $h$ be a left-continuous gauge function.
Let $K$ be a compact metric space such that $\dim_T U\geq n$ and $\iP^h(U)=\infty$ for all
non-empty open sets $U\subset K$. Then for a generic $f\in C(K,\RR^n)$ for all $y\in \inter f(K)$ we have
\[ \iP^h(f^{-1}(y))=\infty. \]
Moreover, the measures $\iP^h|_{f^{-1}(y)}$ are not $\sigma$-finite.
\end{theorem}

\begin{remark} Haase \cite{H} proved that if $\iP^h$ is not $\sigma$-finite on $K$, then
$K$ contains $2^{\aleph_0}$ disjoint compact subsets with non-$\sigma$-finite $\iP^h$ measure.
If $K$ also satisfies the conditions of Theorem~\ref{t:gauge}, then a generic $f\in C(K,\RR^n)$ provides $2^{\aleph_0}$ disjoint fibers of non-$\sigma$-finite $\iP^h$ measure.
\end{remark}

The first part of the next theorem follows from a result of Kato~\cite[Theorem~4.6]{K}. The second part might be known as well, but we could not find it in the literature.
However, it easily follows from Theorem~\ref{t:gauge}.

\begin{theorem} \label{t:alef}
Let $n\in \NN^+$ and let $K$ be a compact metric space with $\dim_T K\geq n$. Then for a generic
$f\in C(K,\RR^n)$ we have
\begin{enumerate}[(i)]
\item $\#f^{-1}(y)\leq n$ if $y\in \partial f(K)$,
\item \label{eq:h=1}  $\# f^{-1}(y)=2^{\aleph_0}$ if $y\in \inter f(K)$.
\end{enumerate}
\end{theorem}

\begin{proof} We only need to prove \eqref{eq:h=1}. By Lemma~\ref{l:int} and
Corollary~\ref{c:R} we may assume that $\dim_T  U\geq n$ for all
non-empty open sets $U\subset K$. Let $h(x)=1$ for all $x\geq 0$.
Applying Theorem~\ref{t:gauge} for $h$ yields that for a generic $f\in C(K,\RR^n)$ for all
$y\in \inter f(K)$ the measures $\iP^h|_{f^{-1}(y)}$ are not $\sigma$-finite.
As $\iP^h=\iP^0$ is the counting measure, each $f^{-1}(y)$ is uncountable.
Then $\# f^{-1}(y)=2^{\aleph_0}$ by \cite[Corollary~6.5]{Ke}, and \eqref{eq:h=1} follows.
\end{proof}

Now we prove Theorem~\ref{t:mainP} based on Theorem~\ref{t:gauge}. We need the following.

\begin{lemma} \label{l:dn}
Let $n\in \NN^+$ and let $K$ be a compact metric space with $\dim_T K\geq n$. If $d<d_P^n(K)$ then there exists a compact set $C\subset K$ such that
for every open set $U$ intersecting $C$ we have $\dim_T (C\cap U)\geq n$ and $\dim_P (C\cap U)>d$.
\end{lemma}

\begin{proof}
Define
\begin{align*} \iW=\{&W\subset K: W \textrm{ is open and there exists  an } F_{\sigma} \textrm{ set } F\subset W  \\
&\textrm{such that } \dim_T F\leq n-1 \textrm{ and } \dim_P (W\setminus F)\leq d \}.
\end{align*}
The countable stability of topological dimension for $F_{\sigma}$ sets and the countable stability of the packing dimension imply that $\iW$ is closed for taking countable unions.
Let $V=\bigcup \iW$, then the Lindel\"of property of $V$ implies that $V\in \iW$. Clearly we have $W\in \iW$ if and only if $W\subset V$. Define $C=K\setminus V$, then $d<d_P^n(K)$ yields that $C$ is a non-empty compact set.
Let $U$ be an open set intersecting $C$, we need to show that $\dim_T (C\cap U)\geq n$ and $\dim_P(C\cap U)>d$. We prove the stronger statement
that there is no $F_{\sigma}$ set $F_1\subset C\cap U$ such that
\begin{equation} \label{eq:F1}
\dim_T F_1\leq n-1 \quad \textrm{and} \quad \dim_P ((C\cap U)\setminus F_1)\leq d.
\end{equation}
Assume to the contrary that there exists such a set $F_1$.
As $U\setminus C\subset V$ is open, we have $U\setminus C\in \iW$.
Thus there exists an $F_{\sigma}$ set $F_2\subset U\setminus C$ such that
\begin{equation} \label{eq:F2}
\dim_T F_2\leq n-1  \quad \textrm{and} \quad  \dim_P ((U\setminus C)\setminus F_2)\leq d.
\end{equation}
Then $F=F_1\cup F_2$ is an $F_{\sigma}$ set in $U$, and the countable stability of
topological dimension for $F_{\sigma}$ sets yields that $\dim_T F\leq n-1$.
The countable stability of packing dimension and the second inequalities
of \eqref{eq:F1} and \eqref{eq:F2} imply that
\[ \dim_P (U\setminus F)=\max\{\dim_P ((C\cap U)\setminus F_1),\,\dim_P ((U\setminus C)\setminus F_2) \}\leq d. \]
Therefore $F$ witnesses $U\in \iW$, so $U\subset V$. Hence $C\cap U=\emptyset$,
which is a contradiction. The proof is complete.
\end{proof}

\begin{proof}[Proof of Theorem~\ref{t:mainP}]
As a countable intersection of co-meager sets is co-meager, it is enough to prove the theorem for a fixed $d<d^n_{P}(K)$.
By Lemma~\ref{l:dn} and Corollary~\ref{c:R} we may assume that $\dim_T U\geq n$ and $\dim_P U>d$
for all non-empty open sets $U\subset K$. Applying Theorem~\ref{t:gauge} for $h(x)=x^d$
implies that for a generic $f\in C(K,\RR^n)$ for all $y\in \inter f(K)$ we have $\iP^d(f^{-1}(y))=\infty$,
so $\dim_P f^{-1}(y)\geq d$. Theorem~\ref{t:int} implies that the open set $U_f=\inter f(K)$ is not empty
for a generic $f\in C(K,\RR^n)$, which finishes the proof.
\end{proof}

In the remaining part of this subsection we prove Theorem~\ref{t:gauge}. First we need some preparation.

\begin{definition} \label{d:D} Assume that a compact metric space $K$ and $n\in \mathbb{N}^+$ are given.
For all $m\in \mathbb{N}^+$ define
\begin{align*} \mathcal{D}_m=\{&f\in C(K,\RR^n): \textrm{ there exists an } \eps>0 \textrm{ such that } \\
&g(K)\setminus B(\partial g(K),1/m)\subset f(K) \textrm{ for all } g\in B(f,\varepsilon)\}.
\end{align*}
If $f\in \mathcal{D}_m$ then there is a witness $\varepsilon(f,m)>0$ corresponding to the definition.
\end{definition}

The following lemma is \cite[Lemma~7.11]{B}.

\begin{lemma} \label{l:D} Let $K$ be a compact metric space in which every non-empty open set is uncountable.
Then $\mathcal{D}_m=\mathcal{D}_m(K,n)$ is dense in $C(K,\RR^n)$ for all $m,n\in \NN^+$.
\end{lemma}

\begin{lemma} \label{l:stab} Let $n\in \NN^+$ and let $K$ be a compact metric space with $\dim_T K\geq n$. Then
for each $y\in \RR^n$ and $r>0$ there exists an onto map $f\colon K\to B(y,r)$ such that $B(y,r-t) \subset g(K)$
for all $0<t<r$ and $g\in B(f,t)$.
\end{lemma}

\begin{proof} By \cite[Theorem~VI~2.]{HW} there exists a continuous map $f\colon K\to \mathbb{R}^n$ with a \emph{stable value},
that is, there exist $y_0\in \mathbb{R}^n$ and $r_0>0$ such that $y_0\in g(K)$ for all $g\in B(f,r_0)$.
We may assume that $y_0=y$ and $r_0=r$ by composing $f$ with affine maps of $\RR^n$.
We show that for all $0<t<r$ and $g\in B(f,t)$ we have $B(y,r-t) \subset g(K)$.
Assume to the contrary that there is a $t\in (0,r)$, a map $g\in B(f,t)$, and $z\in B(y,r-t)\setminus g(K)$.
Let $h(x)=g(x)+(y-z)$. Clearly $h\in B(g, r-t)\subset B(f,r)$, so $y\in h(K)$.
Let $x\in K$ such that $h(x)=y$. Then $g(x)=h(x)-y+z=z$, so $z\in g(K)$, which is a contradiction.

In particular, $B(y,r-t)\subset f(K)$ for all $0<t<r$. As $f(K)$ is closed, we have $B(y,r)\subset f(K)$.
We may assume that $f(K)=B(y,r)$ by replacing $f$ with $p \circ f$, where
$p$ is the radial projection of $\RR^n$ onto $B(y,r)$. This concludes the proof.
\end{proof}

\begin{lemma} \label{l:tech} Let $n\in \NN^+$ and let $\eps>0$.
Assume that $h$ is a left-continuous gauge function and $K$ is a compact metric space such that
$\dim_T U\geq n$ and $\iP^h(U)=\infty$ for every non-empty open set $U\subset K$.
Assume that $\iU\subset C(K,\RR^n)$ is open, and $V\subset \RR^n$ is a non-empty open set such that $V\subset f(K)$ for some $f\in \iU$. Then there is an open set $\iV\subset \iU$ and $N\in \NN^+$, for each $1\leq i\leq N$ there is an
$N_i\in \NN^+$ and an open set $V_i \subset \RR^n$, and for each $1\leq i \leq N$ and $1\leq j\leq N_i$ there is a
non-empty regular closed set $B_{ij}\subset K$ such that
\begin{enumerate}[(i)]
\item  \label{t1} $V=\bigcup_{i=1}^{N} V_{i}$,
\item \label{t1.5} $\diam B_{ij}\leq \eps$ for all $1\leq i \leq N$ and $1\leq j\leq N_i$,
\item \label{t2} $B_{ij}\cap B_{k\ell}=\emptyset$ if $(i,j)\neq (k,\ell)$,
\item \label{t3} $P^h_{\eps} (S_{i})\geq 1/\eps$ whenever $S_i$ meets $B_{ij}$ for all $1\leq j \leq N_{i}$,
\item \label{t4} $V_{i}\subset g(B_{ij})$ and $\diam g(B_{ij})\leq 6\eps$ for all $g\in \iV$, $1\leq i\leq N$, and $1\leq j\leq N_i$.
\end{enumerate}
\end{lemma}

\begin{proof} Fix $f\in \iU$ such that $V\subset f(K)$. We may assume that $U(f,5\eps)\subset \iU$, otherwise we replace
$\eps$ by a smaller positive number. Since $K$ is compact and $f$ is uniformly continuous, there
is a $\delta_1>0$ and there are finitely many distinct points $x_{1},...,x_{N}\in K$ such that
\begin{equation} \label{eq:KB}   K=\bigcup_{i=1}^{N} B(x_{i},\delta_1)
\end{equation}
and for all $i\in \{1,\dots,N\}$ we have
\begin{equation} \label{eq:fB}
f(B(x_i,\delta_1))\subset U(y_i,\eps),
\end{equation}
where $y_i=f(x_i)$. Let $V_i=U(y_i,\eps)\cap V$, then clearly
$ \bigcup_{i=1}^{N} V_i \subset V$. By \eqref{eq:fB} we have
$V\subset f(K)\cap V\subset \bigcup_{i=1}^{N} V_i$, so \eqref{t1} holds.

Choose $0<\delta_2<\delta_{1}$ such that the balls $U(x_{i},\delta_2)$ are disjoint.
Fix an arbitrary $1\leq i\leq N$.
Since $\iP^h(U(x_i,\delta_2))=\infty$, there exists a finite set
\[ Z_{i}=\{z_{ij}: 1\leq j\leq N_i \}\subset U(x_i,\delta_2)\]
such that $P^h_{\eps}(Z_i)> 1/\eps$.
By the left-continuity of $h$ there exists a $0<\delta_3<\eps/2$ such that
$B(z_{ij},\delta_3)$ are disjoint balls in $U(x_i,\delta_2)$
and $P^h_{\eps}(S_i)\geq 1/\eps$ for each set $S_i\subset K$ which meets $B(z_{ij},\delta_3)$ for all $1\leq j\leq N_i$.
For every $j\in \{1,\dots, N_i\}$ define the non-empty regular closed set
\[ B_{ij}=\cl U(z_{ij},\delta_3)\subset B(z_{ij},\delta_3).\]
Then $\diam B_{ij}\leq 2\delta_3<\eps$ yields \eqref{t1.5}. Since the sets
$\{U(x_i,\delta_2): 1\leq i\leq N\}$ and for all $1\leq i\leq N$ the sets
$\{B(z_{ij},\delta_3): 1\leq j \leq N_i\}$ consist of pairwise disjoint balls,
we obtain \eqref{t2}. The definition of $\delta_3$ implies \eqref{t3}.

As $\dim_T B_{ij} \geq n$, by Lemma~\ref{l:stab} there are onto maps
$f_{ij}\colon B_{ij}\to B(y_i,2\eps)$ such that
\begin{equation} \label{eq:Byi} B(y_i,\eps)\subset g_{ij}(B_{ij}) \quad \textrm{for all} \quad  g_{ij}\in B(f_{ij},\eps).
\end{equation}
Tietze's extension theorem and \eqref{eq:fB} yield that for each $i\in \{1,\dots, N\}$
there exists a continuous map $F_i\colon B(x_i,\delta_1)\to B(y_i,2\eps)$ such that $F_i=f_{ij}$ on $B_{ij}$ and $F_i=f$ on
$B(x_i,\delta_1)\setminus U(x_i,\delta_2)$. Let $F(x)=F_i(x)$
for all $x\in B(x_i,\delta_1)$ and $i\in \{1,\dots,N\}$. The construction, \eqref{eq:KB}, and \eqref{eq:fB}
imply that $F\in C(K,\RR^n)$ is well-defined and for all $1\leq i\leq N$ we have
\begin{equation} \label{eq:FB} F(B(x_i,\delta_1))\cup f(B(x_i,\delta_1))=B(y_i,2\eps).
\end{equation}
Clearly \eqref{eq:KB} and \eqref{eq:FB} yield that $F\in B(f,4\eps)$. Define $\iV=U(F,\eps)$, then we obtain $\iV\subset U(f,5\eps)\subset \iU$. Fix $1\leq i\leq N$, $1\leq j\leq N_i$, and $g\in \iV$ arbitrarily.
Then 
\[g(B_{ij})\subset B(f_{ij}(B_{ij}),\eps) \subset B(y_i,3\eps),\]
so $\diam g(B_{ij})\leq 6\eps$. Clearly $g_{ij}=g|_{B_{ij}}$ satisfies $g_{ij}\in B(f_{ij},\eps)$, thus
\eqref{eq:Byi} yields
\[ V_i\subset U(y_i,\eps)\subset g_{ij}(B_{ij})=g(B_{ij}), \]
so \eqref{t4} holds. The proof is complete.
\end{proof}

Now we are ready to prove Theorem~\ref{t:gauge}.

\begin{proof}[Proof of Theorem~\ref{t:gauge}]
For each $m\in \NN^+$ let
\[\iG_m=\{f\in C(K,\RR^n): \iP^h|_{f^{-1}(y)} \textrm{ is not $\sigma$-finite for all } y\in f(K)\setminus B(\partial f(K),1/m)\}.\]
Fix $m\in \NN^+$. It is enough to show that $\mathcal{G}_m$ is co-meager, since then $\iG=\bigcap_{m=1}^{\infty} \iG_m$ will be our desired co-meager set in $C(K,\RR^n)$. We will play the Banach-Mazur game in the complete metric space $C(K,\RR^n)$:
First Player I chooses a non-empty open set $\mathcal{U}_0\subset C(K,\RR^n)$, then Player II
chooses a non-empty open set $\mathcal{V}_0\subset \mathcal{U}_0$, Player I continues with a non-empty open set $\mathcal{U}_1\subset \mathcal{V}_0$, Player II chooses a non-empty open set $\iV_1\subset \iU_1$, and so on. By definition Player II wins this game if $\bigcap_{k=0}^{\infty} \mathcal{V}_k\subset \mathcal{G}_m$.
It is well known that Player II has a winning strategy iff $\mathcal{G}_m$ is co-meager in $C(K,\RR^n)$, see \cite[Theorem~1]{Ox}
or \cite[Theorem~8.33]{Ke}. Thus we need to prove that Player II has a winning strategy.

Assume that the non-empty open set $\iU_0\subset C(K,\RR^n)$ is given. Lemma~\ref{l:D} implies that we can fix
$f_0\in \mathcal{D}_m\cap \iU_0$ and a witness $\varepsilon=\varepsilon(f_0,m)>0$ corresponding to Definition~\ref{d:D}.
Let $\iV_0=B(f_0,\eps)\cap \iU_0$ and let $V_0=\inter f_0(K)$. The definition of $\iD_m$ yields that
for all $f\in \iV_0$ we have
\begin{equation} \label{eq:fbk} f(K)\setminus B(\partial f(K),1/m)\subset V_0.
\end{equation}
If $V_0=\emptyset $ then clearly $\iV_0\subset \iG_m$, and Player II wins independently of the subsequent moves.
Therefore we may assume that $V_0\subset \RR^n$ is a non-empty open set.

Let $\iU_1\subset \iV_0$ be given. By Lemma~\ref{l:tech} there is a non-empty
open set $\iV_1\subset \iU_1$ and
$N_{0}\in \NN^+$ such that for each $1\leq i_1\leq N_{0}$ there is an integer $N_{i_1}\in \NN^+$ and an
open set $V_{i_1} \subset \RR^n$, and for each $1\leq i_1\leq N_{0}$ and $1\leq i_2\leq N_{i_1}$ there is a non-empty regular closed set
$B_{i_1i_2}\subset K$ such that
\[ V_0=\bigcup_{i_1=1}^{N_0} V_{i_1},\]
for all $1\leq i_1\leq N_{0}$ and $1\leq i_2\leq N_{i_1}$ we have
\[ \diam B_{i_1 i_2}\leq 1,\]
we have
\[ B_{i_1i_2}\cap B_{j_1j_2}=\emptyset \quad \textrm{if} \quad (i_1,i_2) \neq (j_1,j_2),\]
for each $1\leq i_1\leq N_{0}$ if $S_{i_1}$ meets $B_{i_1i}$ for all $1\leq i \leq N_{i_1}$, then
\[ P^h_1(S_{i_1})\geq 1,\]
and for all $g\in \iV_1$, $1\leq i_1\leq N_0$, and $1\leq i_2\leq N_{i_1}$ we have
\[ V_{i_1}\subset g(B_{i_1i_2}) \quad \textrm{and}  \quad \diam g(B_{i_1 i_2})\leq 6.\]

We apply induction on $k$.
If $B_{i_1 \dots i_{2k}}\subset B_{i_1 \dots i_{2k-2}}\subset \dots \subset B_{i_1i_2}\subset K$ is
a nested sequence such that $B_{i_1\dots i_{2k}}\neq \emptyset $ and $B_{i_1\dots i_{2\ell}}$ is a regular closed set in the relative topology
of $B_{i_1\dots i_{2\ell-2}}$ for each $1\leq \ell \leq k$ (we adopt the convention $B_{i_1\dots i_0}=K$),
then Fact~\ref{f:reg} implies that $\inter B_{i_1\dots i_{2k}}\neq \emptyset$, thus $\dim_T B_{i_1\dots i_{2k}}\geq n$.
Therefore, if the pairwise disjoint closed sets $B_{i_1\dots i_{2k}}\subset K$ are given for which the above nested sequences exist,
we can apply Lemma~\ref{l:tech} for them. This yields closed sets $B_{i_1\dots i_{2k+2}}\subset B_{i_1\dots i_{2k}}$ and
non-empty open sets $\iV_{i_1\dots i_{2k+2}}\subset C(B_{i_1\dots i_{2k+2}},\RR^n)$. Then
by Tietze's extension theorem we can define the non-empty open set
\[ \iV_{k+1}=\iU_{k+1} \cap \{f\in C(K,\RR^n): f|_{B_{i_1\dots i_{2k+2}}} \in \iV_{i_1\dots i_{2k+2}} \textrm{ for all } i_1,\dots, i_{2k+2}\}.\]
That is, in the $(k+1)$st step we can define a non-empty open set $\iV_{k+1}\subset \iU_{k+1}$,
for each $1\leq i_{2k+1}\leq N_{i_1\dots i_{2k}}$ a non-empty open set $V_{i_1\dots i_{2k+1}}\subset \RR^n$ and a
positive integer $N_{i_1\dots i_{2k+1}}$, for every $1\leq i_{2k+1}\leq N_{i_1\dots i_{2k}}$ and $1\leq i_{2k+2}\leq N_{i_1\dots i_{2k+1}}$ a
set $B_{i_1\dots i_{2k+2}}\subset B_{i_1\dots i_{2k}}$ such that $B_{i_1\dots i_{2k+2}}$ is a non-empty regular closed set in the relative topology of $B_{i_1 \dots i_{2k}}$ and the following holds. For all $1\leq i_{2k}\leq N_{i_1\dots i_{2k-1}}$ we have
\begin{equation} \label{eqk1}
V_{i_1\dots i_{2k-1}}=\bigcup_{i_{2k+1}=1}^{N_{i_1\dots i_{2k}}} V_{i_1\dots i_{2k+1}},
\end{equation}
for all $1\leq i_{2k+1}\leq N_{i_1\dots i_{2k}}$ and $1\leq i_{2k+2}\leq N_{i_1\dots i_{2k+1}}$ we have
\begin{equation} \label{eqk3}
\diam B_{i_1\dots i_{2k+2}}\leq 2^{-k},
\end{equation}
we have
\begin{equation} \label{eqk5}
B_{i_1\dots i_{2k+2}}\cap B_{j_1\dots j_{2k+2}}=\emptyset \quad \textrm{if} \quad (i_1,\dots, i_{2k+2})\neq (j_1,\dots, j_{2k+2}),
\end{equation}
for each $1\leq i_{2k+1}\leq N_{i_1\dots i_{2k}}$ if $S_{i_1\dots i_{2k+1}}$ meets $B_{i_1\dots i_{2k+1}i}$ for all $1\leq i \leq N_{i_1\dots i_{2k+1}}$, then
\begin{equation} \label{eqk4}
P^h_{2^{-k}} (S_{i_1\dots i_{2k+1}})\geq 2^k,
\end{equation}
and for all $g\in \iV_{k+1}$, $1\leq i_{2k+1}\leq N_{i_1\dots i_{2k}}$, and $1\leq i_{2k+2}\leq N_{i_1\dots i_{2k+1}}$ we have
\begin{equation} \label{eqk2}
V_{i_1\dots i_{2k+1}}\subset g(B_{i_1\dots i_{2k+2}}) \quad \textrm{and} \quad \diam g(B_{i_1\dots i_{2k+2}})\leq 6\cdot 2^{-k}.
\end{equation}

The regularity of the closed sets $B_{i_1\dots i_{2k}}$ and \eqref{eqk5} are only needed to apply the induction, we will not use them later.

Fix $f\in \bigcap_{k=0}^{\infty} \iV_k$, we need to show that $f\in \iF$. Fix $y\in V_0$,
by \eqref{eq:fbk} it is enough to show that $\iP^h|_{f^{-1}(y)}$ is not $\sigma$-finite. Define the set
\begin{equation*}
C=\bigcap_{k=1}^{\infty} C_k, \quad \textrm{where} \quad  C_k=\bigcup \{ B_{i_1\dots i_{2k}}: y\in V_{i_1\dots i_{2k-1}}\}.
\end{equation*}
Then $\eqref{eqk2}$ implies that $f(x)=y$ for all $x\in C$, therefore $C\subset f^{-1}(y)$. Hence it is enough to prove that $\iP^h|_{C}$ is not $\sigma$-finite. Fix an arbitrary open set $U\subset K$ intersecting $C$, by Lemma~\ref{l:pack1} it is enough to prove that $P_{0}^h(C\cap U)=\infty$. 
By \eqref{eqk3} we can fix $k\in \NN^+$ and indexes $(i_1,\dots, i_{2k})$ such that $y\in V_{i_1\dots i_{2k-1}}$ and $B_{i_1\dots i_{2k}}\subset U$. Hence it is sufficient to show that $P_{0}^h(C\cap B_{i_1\dots i_{2k}})=\infty$. Fix an arbitrary
$\ell \geq k$, it is enough to prove that
\begin{equation} \label{eql} P^h_{2^{-\ell}}(C\cap B_{i_1\dots i_{2k}})\geq 2^{\ell}.
\end{equation}
By \eqref{eqk1} we can fix indexes $i_{2k+1},\dots, i_{2\ell+1}$ such that $y\in V_{i_1\dots i_{2\ell+1}}$.
The definition of $C$ and \eqref{eqk1} imply that $C\cap B_{i_1\dots i_{2\ell+1}i}\neq \emptyset$ for all $1\leq i\leq N_{i_1\dots i_{2\ell+1}}$.
Let
\[ S_{i_1\dots i_{2\ell+1}}=\bigcup \{C\cap B_{i_1\dots i_{2\ell+1}i}: 1\leq i\leq N_{i_1\dots i_{2\ell+1}}\}.\]
Then $S_{i_1\dots i_{2\ell+1}}\subset C\cap B_{i_1\dots i_{2k}}$, so \eqref{eqk4} yields that
\[ P^h_{2^{-\ell}}(C\cap B_{i_1\dots i_{2k}})\geq P^h_{2^{-\ell}}(S_{i_1\dots i_{2\ell+1}})\geq 2^{\ell}.\]
Thus \eqref{eql} holds, and the proof is complete.
\end{proof}

\section{Fibers on fractals} \label{s:fractal}

The main goal of this section is to prove Theorem~\ref{t:fractal1}.
Recall the following notation.

\begin{notation} Let $n\in \NN^+$ and let $K$ be a compact metric space. Let
$\dim_{*}$ be one of $\dim_T$, $\dim_H$, or $\dim_P$. For each $f\in C(K,\RR^n)$ let
\[ R_{*}(f)=\{y\in f(K): \dim_{*} f^{-1}(y)=d^n_{*}(K)\}.\]
\end{notation}

\begin{theorem} \label{t:fractal1} Let $n\in \NN^+$ and assume that $\dim_{*}$ is one of $\dim_T$, $\dim_H$, or $\dim_P$.
Let $K$ be a compact metric space with $d^n_{*}(K)>0$. The following are equivalent:
\begin{enumerate}[(i)]
\item \label{i} $ R_{*}(f)=\inter f(K)$ for a generic $f\in C(K,\RR^n)$;
\item \label{ii} $R_{*}(f)$ is dense in $\inter f(K)$ for a generic $f\in C(K,\RR^n)$;
\item \label{iii} $d^n_{*}(U)=d^n_{*}(K)$ or $\dim_T U<n$ for all open sets $U\subset K$.
\end{enumerate}
\end{theorem}

We obtain the following corollary.

\begin{corollary} \label{c:fractal2}
Let $n\in \NN^+$ and assume that $\dim_{*}$ is one of $\dim_T$, $\dim_H$, or $\dim_P$.
Let $K$ be a compact metric space with $d^n_{*}(K)>0$. The following are equivalent:
\begin{enumerate}[(1)]
\item \label{1} $R_{*}(f)=\inter f(K)$ and $\inter f(K)$ is dense in $f(K)$ for a generic $f\in C(K,\RR^n)$;
\item \label{2} $R_{*}(f)$ is dense in $f(K)$ for a generic $f\in C(K,\RR^n)$;
\item \label{3} $d^n_{*}(U)=d^n_{*}(K)$ for all non-empty open sets $U\subset K$.
\end{enumerate}
\end{corollary}

Let $m\geq n$ be positive integers. Theorem~\ref{t:dtn} implies that
$d_T^n([0,1]^m)=m-n$, and Lemmas~\ref{l:prodH} and~\ref{l:prodP} yield
that $d^n_H([0,1]^m)=m-n$ and $d^n_{P}([0,1]^m)=m$.
Thus Theorem~\ref{t:fractal1} implies the following extension of
Theorem~\ref{t:Kirchheim}.

\begin{corollary} \label{c:Ki} Let $m,n\in \NN^+$ with $m\geq n$.
For a generic $f\in C([0,1]^m,\RR^n)$ for all $y\in \inter f([0,1]^m)$ we have
\[ \dim_T f^{-1}(y)=\dim_H f^{-1}(y)=m-n \quad \textrm{ and } \quad \dim_P f^{-1}(y)=m.\]
\end{corollary}

Before proving Theorem~\ref{t:fractal1} we need some preparation.
The proof of $\eqref{iii} \Longrightarrow \eqref{i}$ will basically follow
the proof of \cite[Theorem~7.9]{B}, where the new ingredient will be the application of Lemma~\ref{l:iF}.
First we need some lemmas.

\begin{lemma} \label{l:mm} Let $X,Z$ be Polish spaces, and let $E\subset
X\times Z$ be a Borel set. Assume that $E_{x}$ is compact for all
$x\in X$. Let $\dim_{*}$ be one of $\dim_T$, $\dim_H$, or $\dim_P$. Define
\[ d_{*}\colon X\rightarrow [-1,\infty], \quad d_{*}(x)=\dim_{*} E_{x}.\]
Then $d_{*}$ is measurable for $\sigma(\mathbf{A})$.  \end{lemma}

\begin{proof} By \cite[Theorem~6.1]{MM} the map $d_H$ is Borel measurable, and
\cite[Theorem~6.4]{MM} yields that $d_P$ is measurable for $\sigma(\mathbf{A})$. Thus it is enough to
prove that $d_T$ is Borel measurable. Define
\[ D_T\colon \iK(K)\to [-1,\infty], \quad D_T(C)=\dim_T C,\]
where $\iK(K)$ is the set of compact subsets of
$K$ endowed with the Hausdorff metric, see \cite[Section~4.F]{Ke} for the definition.
By \cite[page~108,~Theorem~4]{Ku} the map $D_T$ is Borel measurable.
Define $S\colon X\to \iK(K)$ as $S(x)=E_x$. Then $S$ is Borel measurable
by \cite[Theorem~28.8]{Ke}. Therefore the map $d_T=D_{T} \circ S$ is Borel measurable, too.
\end{proof}

\begin{remark} Unlike \cite{MM}, we adopt the convention $\dim_{*}\emptyset=-1$,
which modifies some fibers of $d_{*}$ by $\{x\in X: E_x=\emptyset\}=(\pr_{X} (E))^{c}$,
where $\pr_{X}$ denotes the natural projection of $X\times Z$ onto $X$. In order to fix this problem
it is enough to show that $\pr_{X} (E)$ is Borel. Since $E$ is Borel and
$E_{x}$ is compact for every $x\in X$, this follows from \cite[Theorem~18.18]{Ke}.
\end{remark}

Now assume that a compact metric space $K$ and the numbers $n\in \NN^+$ and $d\in \RR$ are given.
Let $\dim_{*}$ be one of $\dim_T$, $\dim_H$, or $\dim_P$. The proof of the following lemma is analogous to that of \cite[Lemma~2.11]{BBE2}.

\begin{lemma} \label{l:Baire}
The set
\[ \Delta=\left\{(f,y)\in C(K,\RR^n)\times \mathbb{R}^n: \dim_{*} f^{-1}(y)<d\right\} \]
is in $\sigma(\mathbf{A})$.
\end{lemma}

\begin{proof} Let $X=C(K,\RR^n)\times \RR^n$, let $Z=K$, and define
\[ E= \{(f,y,z)\in C(K,\RR^n)\times \RR^n\times K: f(z)=y\}\subseteq X\times Z. \]
Clearly $X,Z$ are Polish spaces and $E$ is closed, thus Borel.
For each $(f,y)\in X$ the set $E_{(f,y)}=\{z\in K: f(z)=y\}=f^{-1}(y)$ is compact. Thus
Lemma~\ref{l:mm} yields that the map
\[ d_{*}\colon X\rightarrow [0,\infty], \quad d_{*}((f,y))=\dim_{*}E_{(f,y)}=\dim_{*} f^{-1}(y) \]
is measurable for $\sigma(\mathbf{A})$. Therefore
\[ d_{*}^{-1}\left((-\infty, d)\right)=\left\{(f,y)\in C(K,\RR^n)\times \RR^n: \dim_{*} f^{-1}(y)<d \right\}=\Delta \]
is in $\sigma(\mathbf{A})$.
\end{proof}

\begin{definition} \label{d:H}
For all $0<r_1< r_2$ and $y_0\in \RR^n$ define
\begin{align*} \iH(r_1,r_2,y_0)=\{&f\in C(K,\RR^n): f(K)\subset B(y_0,r_2) \textrm{ and } \\
 &\dim_{*} f^{-1}(y)\geq d \textrm{ for all } y\in B(y_0,r_1)\}.
\end{align*}
\end{definition}

\begin{lemma} \label{l:FBaire} The sets $\iH(r_1,r_2,y_0)$ have the Baire property.
\end{lemma}

\begin{proof} Let $\iH=\iH(r_1,r_2,y_0)$ and let $\pr\colon C(K,\RR^n)\times \RR^n\to C(K,\RR^n)$ be the natural
projection of $C(K,\RR^n)\times \RR^n$ onto $C(K,\RR^n)$. Define
\[ \Delta=\left\{(f,y)\in C(K,\RR^n)\times \mathbb{R}^n: \dim_{*} f^{-1}(y)<d\right\}, \]
by Lemma~\ref{l:Baire} the set $\Delta$ is in $\sigma(\mathbf{A})$. It is easy to see that $\iH=\iH_1\cap \iH_2$, where
\begin{align*} \iH_1&=\{f\in C(K,\RR^n): f(K)\subset B(y_0,r_2)\}, \\
\iH_2&=\left(\pr\left(\Delta \cap (C(K,\RR^n)\times B(y_0,r_1)) \right)\right)^c.
\end{align*}
Clearly $\iH_1$ is closed. As $\sigma(\mathbf{A})$ is closed under taking intersection, projection, and
complement, we obtain that $\iH_2$ is in $\sigma(\mathbf{A})$. Therefore $\iH$ is in $\sigma(\mathbf{A})$, so it has the Baire property.
\end{proof}

\begin{lemma} \label{l:iF}
If $\iH(r_1,r_2,y_0)$ is of second category for some parameters $r_1,r_2,y_0$, then it is of second category for all parameters.
\end{lemma}

\begin{proof}
Assume that $0<r_1<r_2$, $0<q_1<q_2$, and $y_0,z_0\in \RR^n$. We need to prove that if $\iH(r_1,r_2,y_0)$ is of second category, then
$\iH(q_1,q_2,z_0)$ is of second category, too. Clearly it is enough to find a homeomorphism $H\colon C(K,\RR^n)\to C(K,\RR^n)$
such that $H(\iH(r_1,r_2,y_0))=\iH(q_1,q_2,z_0)$. Define a homeomorphism $h\colon \RR^n\to \RR^n$ such that
$h(B(y_0,r_i))=B(z_0,q_i)$ for $i\in \{1,2\}$, and $h$ is affine on $\RR^n\setminus B(y_0,r_2)$.
Then $h$ and $h^{-1}$ are uniformly continuous, so
\[ H\colon C(K,\RR^n)\to C(K,\RR^n), \quad H(f)=h\circ f \]
is a homeomorphism with $H(\iH(r_1,r_2,y_0))=\iH(q_1,q_2,z_0)$.
\end{proof}

Now we are ready to prove Theorem~\ref{t:fractal1}.

\begin{proof}[Proof of Theorem~\ref{t:fractal1}]
The implication $\eqref{i} \Longrightarrow \eqref{ii}$ is straightforward.

Now we prove $\eqref{ii} \Longrightarrow \eqref{iii}$. Assume to the contrary that there exists an open set $U\subset K$
such that $\dim_T U\geq n$ and $d^n_{*}(U)<d^n_{*}(K)$. As $U$ is a countable union of compact sets,
the countable stability of topological dimension for closed
sets yields that there is a compact set $C\subset U$ such that $\dim_T C\geq n$.
Let $D$ be a compact set such that $C\subset \inter D\subset D\subset U$. Then clearly $d^n_{*}(D)<d^n_{*}(K)$.
Define the sets
\begin{align*}
\iA&=\{f\in C(K,\RR^n): \inter f(C)\neq \emptyset\}, \\
\iB&=\{f\in C(K, \RR^n): \dim_{*}(f^{-1}(y)\cap D)\leq d^{n}_{*}(D) \textrm{ for all } y\in \RR^n\}, \\
\iC&=\{f\in C(K,\RR^n): f(C)\cap f(K\setminus \inter D)=\emptyset\}.
\end{align*}
As $\dim_T C\geq n$, Theorem~\ref{t:int} and Corollary~\ref{c:R} imply that $\iA$ is co-meager.
The Main Theorem and Corollary~\ref{c:R} yield
that $\iB$ is co-meager, too. Tietze's extension theorem
implies that there are distinct points
$y_1,y_2\in \RR^n$ and there is a continuous map $f\in C(K,\RR^n)$ such that $f(C)=y_1$ and $f(K\setminus \inter D)=y_2$.
Let $r=|y_1-y_2|/2>0$, then clearly $U(f,r)\subset \iC$,
so $\iC$ is of second category. Define
\[ \iF=\iA\cap \iB \cap \iC, \]
then $\iF$ is of second category. Let $f\in \iF$ and let $U_f=\inter f(C)\subset \inter f(K)$.
Since $f\in \iA$, the set $U_f$ is non-empty open in $\inter f(K)$.
As $f\in \iC$, we have
$f^{-1}(y)\subset D$ for all $y\in U_f$. Therefore $f\in \iB$ implies that for all $y\in U_f$ we have
\begin{equation*}
\dim_{*} f^{-1}(y)=\dim_{*} (f^{-1}(y)\cap D)\leq d^{n}_{*}(D)< d^{n}_{*}(K).
\end{equation*}
Thus $R_{*}(f)\cap U_f=\emptyset$, so $R_{*}(f)$ is not dense in $\inter K$.
This contradicts \eqref{ii}, so the implication $\eqref{ii} \Longrightarrow \eqref{iii}$ follows.

Finally, we show that $\eqref{iii} \Longrightarrow \eqref{i}$. For a generic $f\in C(K,\RR^n)$ we obtain $R_{*}(f)\subset \inter f(K)$ by Theorem~\ref{t:alef} and $\dim_{*} f^{-1}(y)\leq d^n_{*}(K)$ for all $y\in \mathbb{R}^n$ by the Main Theorem.
Thus it is enough to prove that $\dim_{*} f^{-1}(y)\geq d_{*}^n(K)$ for a generic $f\in C(K,\RR^n)$ for all
 $y\in \inter f(K)$. Let $V\subset K$ be the maximal open set with $\dim_T V<n$. Clearly $d^{*}(A\setminus V)=d^{*}(A)$ for each $A\subset K$, so by Lemma~\ref{l:int} and Corollary~\ref{c:R} we may assume that $V=\emptyset$. Thus $\dim_T U\geq n$ for all non-empty open sets $U\subset K$.
Choose a sequence $d_m\nearrow d^n_{*}(K)$. Let us fix $m\in \mathbb{N}^{+}$ and define
\begin{equation*} \mathcal{G}_{m} = \{f\in C(K,\RR^n): \dim_{*} f^{-1}(y)\geq d_m \textrm{ for all } y\in f(K)\setminus B(\partial f(K),1/m)\}.
\end{equation*}

It is sufficient to verify that $\mathcal{G}_m$ is co-meager, since then
$\iG=\bigcap_{m=1}^{\infty} \iG_m$ will be our desired co-meager set in $C(K,\RR^n)$.
 As $\dim_T U\geq n$ for every non-empty open set $U\subset K$,
we can apply Lemma~\ref{l:D}. Let us fix an arbitrary
$f_0\in \mathcal{D}_m$ and a witness $\varepsilon=\varepsilon(f_0,m)>0$ corresponding to Definition~\ref{d:D}.
As $\mathcal{D}_m$ is dense in $C(K,\RR^n)$ by Lemma~\ref{l:D}, it is enough to show
that $\mathcal{G}_m \cap B(f_{0},\varepsilon)$ contains a set which is of second category and has the Baire property.

Since $K$ is compact and $f_0$ is uniformly continuous,
there is a $\delta_1>0$ and there are finitely many distinct points $x_{1},...,x_{k}\in K$
such that
\begin{equation}\label{*cover*}
K=\bigcup_{i=1}^{k}B(x_{i},\delta_1)
\end{equation}
and for each $i$ the oscillation of $f_{0}$ on
$B(x_{i},\delta_1)$ is less than $\varepsilon/4$. 
Then clearly
\begin{equation} \label{eq:cover} f_0(K)\subset \bigcup_{i=1}^k B(f_0(x_i),\eps/4).
\end{equation}
Let us choose $0<\delta_2<\delta_1$ such that the balls $K_i=B(x_{i},\delta_2)$ are disjoint and fix $i\in \{1,\dots,k\}$. According to Definition~\ref{d:H} for all $0<r_1< r_2$ and $y_0\in \RR^n$ let
\begin{align*} \iH_i(r_1,r_2,y_0)=\{&f\in C(K_i,\RR^n): f(K_i)\subset B(y_0,r_2) \textrm{ and } \\
 &\dim_{*} f^{-1}(y)\geq d_m \textrm{ for all } y\in B(y_0,r_1)\},
\end{align*}
and define
\[ \iH_i=\iH_i(\eps/4,\eps/2,f_0(x_i)). \]
Lemma~\ref{l:FBaire} yields that $\iH_i$ has the Baire property. Now we prove that
$\iH_i$ is of second category. As $\dim_T K_i\geq n$ and $d^n_{*}(K_i)=d^n_{*}(K)$, the Main Theorem implies that
for a generic $f\in C(K_i,\RR^n)$ there exists a non-empty open set $U_{f,d_m}\subset \RR^n$
such that $\dim_{*} f^{-1}(y)\geq d_m$ for all $y\in U_{f,d_m}$. Thus
Baire's category theorem yields that there exist $0<r_1<r_2$ and $y_0\in \RR^n$ such that $\iH_i(r_1,r_2,y_0)$ is of second category.
Then Lemma~\ref{l:iF} implies that $\iH_i$ is of second category.

Clearly $\iH_i\subset B(f_0|_{K_i},\eps)$. Set
\[ \iF=\bigcap_{i=1}^k \iF_i, \quad  \textrm{where} \quad \iF_i=\{f\in B(f_0,\eps): f|_{K_i}\in \iH_i\}. \]
As the sets $\iH_i$ have the Baire property, $\iF$ has the Baire property, too.
Clearly $\mathcal{F}\subset B(f_{0},\varepsilon)$, and repeating the proof of
\cite[Lemma~3.8]{BBE2} verbatim yields that $\mathcal{F}$ is of second category.
Therefore, it is enough to show that $\mathcal{F}\subset \mathcal{G}_m$. Assume that
$f\in \mathcal{F}$ and $y\in f(K)\setminus B(\partial f(K),1/m)$,
we need to prove that $\dim_{*} f^{-1}(y)\geq d_m$.
The definition of $\varepsilon=\varepsilon(f_0,m)$ and $f\in B(f_0,\varepsilon)$ yield
that $y\in f_{0}(K)$. By \eqref{eq:cover} there exists $i\in \{1,\dots,k\}$ such
that $y \in B(f_{0}(x_{i}),\eps/4)$. Then $f|_{K_i}\in \iH_i$ implies that
$\dim_{*} f^{-1}(y)\geq \dim_{*}(f|_{K_i})^{-1}(y)\geq d_m$. This completes the proof.
\end{proof}

\begin{remark} In the case of packing dimension the implication
$\eqref{iii} \Longrightarrow \eqref{i}$ simply follows from
Theorem~\ref{t:gauge}. However,
Lemma~\ref{l:Baire} is not superfluous even
in the case of packing dimension, we will use it to prove Theorem~\ref{t:max}.
\end{remark}

\begin{proof}[Proof of Corollary~\ref{c:fractal2}]
The implication $\eqref{1} \Longrightarrow \eqref{2}$ is straightforward.

Now we prove
$\eqref{2}  \Longrightarrow \eqref{3}$. Assume to the contrary that there is a non-empty open set
$U\subset K$ such that $d^n_{*}(U)<d^n_{*}(K)$. Theorem~\ref{t:fractal1} implies that $\dim_T U<n$.
Let $C,D$ be non-empty compact sets such that $C\subset \inter D\subset D\subset U$. Define
\begin{align*}
\iA&=\{f\in C(K,\RR^n): \dim_{*} (f^{-1}(y) \cap D)\leq 0 \textrm{ for all } y\in \RR^n\}, \\
\iB&=\{f\in C(K,\RR^n): f(C)\cap f(K\setminus \inter D)=\emptyset\}.
\end{align*}
Theorem~\ref{t:Hurewicz} and Corollary~\ref{c:R} imply that $\iA$ is prevalent.
Repeating the arguments of the proof Theorem~\ref{t:fractal1} yields that $\iB$ is of second category, so
$\iA\cap \iB$ is of second category, too. Fix $f\in \iA\cap \iB$ and let $U_f=f(K)\setminus f(K\setminus \inter D)$. Clearly
$U_f$ is relatively open in $f(K)$, and $f\in \iB$ implies that $U_f\neq \emptyset$. Let $y\in U_f$ be arbitrarily fixed. Clearly $f^{-1}(y)\subset D$, so $f\in \iA$ yields that
\[ \dim_{*} f^{-1}(y)=\dim_{*} (f^{-1}(y)\cap D)=0<d^n_{*}(K). \]
Thus $R_{*}(f)\cap U_f=\emptyset$, so $R_{*}(f)$ is not dense in $f(K)$, which contradicts \eqref{2}.

Finally, we show $\eqref{3} \Longrightarrow \eqref{1}$. By Theorem~\ref{t:fractal1} it is enough to prove that
$\inter f(K)$ is dense in $f(K)$ for a generic $f\in C(K,\RR^n)$. Let $\iD$ be a countable dense subset of $K$ and let
$\iB=\{B(x,1/i): x\in \iD,~i\in \NN^+\}$. Let $B\in \iB$ be given. As $d^n_{*}(B)=d^n_{*}(K)$,
Fact~\ref{f:equiv} yields that $\dim_T B\geq n$. Therefore
Theorem~\ref{t:int} and Corollary~\ref{c:R} imply that $\inter f(B)\neq \emptyset $ for a generic $f\in C(K,\RR^n)$.
As a countable intersection of co-meager sets is co-meager, for a generic $f\in C(K,\RR^n)$ for all $B\in \iB$ we have
$\inter f(B)\neq \emptyset$, so $\inter f(K)$ is dense in $f(K)$. The proof is complete.
\end{proof}

\section{Dimensions of the boundary of generic images} \label{s:partial}

The goal of this section is to prove the following theorem.

\begin{theorem} \label{t:bound} Let $n \in  \NN^+$ and let $K$ be a compact metric space with $\dim_T K\geq n$.
Then for a generic $f\in C(K,\RR^n)$ we have
\[ \dim_T \partial f(K)=\dim_H \partial f(K)=n-1. \]
Moreover, let $h$ be a gauge function with $\lim_{r\to 0+} h(r)/r^{n-1}=0$.
Then for a generic $f\in C(K,\RR^n)$ we have $\iH^h(\partial f(K))=0$ and $\iH^{n-1}(\partial f(K))>0$.
\end{theorem}

\begin{notation} Let $n\in \NN^+$ and let $K$ be a compact metric space. Let
$\dim_{*}$ be one of $\dim_T$, $\dim_H$, or $\dim_P$. For each $f\in C(K,\RR^n)$ let
\[ S_{*}(f)=\{y\in f(K): \dim_{*} f^{-1}(y)<d^n_{*}(K)\}. \]
\end{notation}

Theorems~\ref{t:fractal1} and~\ref{t:bound} imply the following.

\begin{corollary} \label{c:alt} Let $n\in \NN^+$ and assume that $\dim_{*}$ is one of $\dim_T$, $\dim_H$, or $\dim_P$.
Let $K$ be a compact metric space with $d^n_{*}(K)>0$. Exactly one of the following holds:
\begin{enumerate}[(a)]
\item $\dim_H S_{*}(f)=n-1$ for a generic $f\in C(K,\RR^n)$;
\item  $\inter S_{*} (f)\neq \emptyset$ for a generic $f\in C(K,\RR^n)$.
\end{enumerate}
Moreover, \eqref{a1} is equivalent to the statements of Theorem~\ref{t:fractal1}.
\end{corollary}

First we need some preparation. The next theorem is \cite[Theorem~2.1]{BFFH}, which generalizes
Theorem~\ref{t:kato}.

\begin{theorem}[Balka--Farkas--Fraser--Hyde]\label{t:BFFH}
Let $n\in \NN^+$ and let $K$ be a compact metric space. Then for a generic $f\in C(K,\RR^n)$ we have
\[ \dim_T f(K)=\dim_H f(K)=\min\{\dim_T K, n\}. \]
\end{theorem}

We need the following better upper bound.

\begin{theorem} \label{t:better} Let $m,n\in \NN^+$ with $m<n$. Let
$K$ be a compact metric space and let $h$ be a gauge function with $\dim_T K=m$ and $\lim_{r\to 0+} h(r)/r^m=0$. Then for a generic
$f\in C(K,\RR^n)$ we have
\[ \iH^h(f(K))=0. \]
\end{theorem}

\begin{proof}
Let $g\colon [0,\infty)\to [0,\infty)$ be the right-continuous modification of
$h$ defined as $g(r)=\lim_{t\to r+} h(t)$. Clearly $g$ is non-decreasing and
$\lim_{r\to 0+} h(r)/r^m=0$ implies that $g(0)=0$ and $\lim_{r\to 0+} g(r)/r^m=0$.
As $g(r)\geq h(r)$ for all $r\geq 0$, it is enough to prove that $\iH^g(f(K))=0$
for a generic $f\in C(K,\RR^n)$. Let
\[ \iF=\{f\in C(K, \RR^n): \iH^g(f(K))=0\}, \]
and for all $i\in \NN^+$ define
\begin{align*} \iF_i=\{&f \in C(K,\RR^n): \textrm{ there are open sets } U_1,\dots,U_k\subset \RR^n \\
& \left.\textrm{such that } f(K)\subset \bigcup_{j=1}^k U_j \textrm{ and } \sum_{j=1}^k g(\diam U_j)<1/i\right\}.
\end{align*}
The sets $\iF_i$ are clearly open. As $K$ is compact and $g$ is right-continuous, we obtain
$\iF=\bigcap_{i=1}^{\infty} \iF_i$. Thus $\iF$ is $G_{\delta}$, so it is enough to prove that
$\iF$ is dense in $C(K,\RR^n)$.
The set
\[ \iG=\{f\in C(K,\RR^n): f(K) \textrm{ is contained in an $m$-dimensional polyhedron} \} \]
is dense in $C(K,\RR^n)$, see \cite[Chapter~1.10]{E2}. Clearly $\iG\subset \iF$, hence
$\iF$ is dense in $C(K,\RR^n)$, too. The proof is complete.
\end{proof}

The hard part of Theorem~\ref{t:bound} is the following.

\begin{theorem} \label{t:M} Let $n\in \NN^+$ and let
$K$ be a compact metric space such that we have $\dim_T U\geq n$ for all non-empty open sets $U\subset K$.
Let $h$ be a gauge function with $\lim_{r\to 0+} h(r)/r^{n-1}=0$. Then for a generic $f\in C(K,\RR^n)$ we have
\[ \iH^h(\partial f(K))=0.\]
\end{theorem}

First we deduce Theorem~\ref{t:bound} from Theorems~\ref{t:better} and~\ref{t:M}.

\begin{proof}[Proof of Theorem~\ref{t:bound}] We prove the lower bounds first.
Theorem~\ref{t:int} yields that for a generic $f\in C(K,\RR^n)$ we have $\inter f(K)\neq \emptyset$.
First fix such $f\in C(K,\RR^n)$.
By \cite[Theorem~1.8.12]{E} the boundary of a bounded, non-empty open set in $\RR^n$ has topological dimension at least $n-1$,
so $\dim_T \partial (\inter f(K))\geq n-1$. Clearly $\partial (\inter f(K))\subset \partial f(K)$, so
$\dim_T \partial f(K)\geq n-1$. Let $\pr\colon \RR^n \to \RR^{n-1}$ denote the orthogonal projection of
$\RR^n$ onto its first $n-1$ coordinates (where $\RR^0=\{0 \}$ by convention),
then $\pr(f(K))$ contains a non-empty open set, and clearly we have $\pr(\partial f(K))=\pr(f(K))$.
As Hausdorff measures cannot increase under projections, we obtain
\[ \iH^{n-1}(\partial f(K))\geq \iH^{n-1}(\pr(\partial f(K)))=\iH^{n-1}(\pr(f(K)))>0. \]

Now we prove the upper bounds. It is enough to show that $\iH^h(\partial f(K))=0$ for a generic $f\in C(K,\RR^n)$, then
the choice $h(x)=x^{n-1}/\log(1/x)$ and Theorem~\ref{t:<} will immediately imply that
\[ \dim_T \partial f(K)\leq \dim_H \partial f(K)\leq n-1. \]
Let $V\subset K$ be the maximal open set such that $\dim_T V<n$ and let $C=K\setminus V$.
By Lemma~\ref{l:int} we have $\dim_T (C\cap U)\geq n$ for all open sets $U\subset K$ intersecting $C$.
Assume that $V=\bigcup_{i=1}^{\infty} K_i$, where $K_i\subset K$ are compact.
As $\dim_T K_i\leq n-1$ and a countable intersection of co-meager sets is co-meager,
Theorem~\ref{t:better}, Theorem~\ref{t:M}, and Corollary~\ref{c:R} imply that for a generic
$f\in C(K,\RR^n)$ we have $\iH^h(f(K_i))=0$ for all $i\geq 1$ and $\iH^h(\partial f(C))=0$.
Clearly
\[ \partial f(K)\subset \bigcup_{i=1}^{\infty} f(K_i) \cup \partial f(C), \]
so the countably subadditivity of $\iH^h$ implies that $\iH^h(\partial f(K))=0$ for a generic $f\in C(K,\RR^n)$.
The proof is complete.
\end{proof}

Before proving Theorem~\ref{t:M} we need some more preparation.

\begin{definition} \label{d:GnK} For a compact metric space $K$ and $n\in \NN^+$ define
\begin{align*} \iG_n(K)=&\left\{g\in C(K,\RR^n): g(K)=\bigcup_{i=1}^{k} B(y_i,r) \textrm{ for some } k,~y_i\in \RR^n,  \textrm{ and }  r>0 \right. \\
& \left. ~~ \textrm {such that } \bigcup_{i=1}^{k} B(y_i,r-t)\subset f(K) \textrm{ for all }  0<t<r \textrm{ and }  f\in B(g,t) \right\}.
\end{align*}
\end{definition}

The proof of the following lemma is similar to that of Lemma~\ref{l:tech}.

\begin{lemma} \label{l:GnK} Let $n\in \NN^+$ and let $K$ be a compact metric space such that
we have $\dim_T U\geq n$ for all non-empty open sets $U\subset K$.
Then $\iG_n(K)$ is dense in $C(K,\RR^n)$.
\end{lemma}

\begin{proof} Assume that $f_0\in C(K,\RR^n)$ and $r>0$ are given,
we need to show that $\iG_n(K)\cap B(f_0,2r)\neq \emptyset$.
Since $K$ is compact and $f_{0}$ is uniformly continuous, there
is a $\delta>0$ and there are finitely many distinct points $x_{1},...,x_{k}\in K$ such that
\begin{equation} \label{eq:K=}  K=\bigcup_{i=1}^{k}B(x_{i},\delta)
\end{equation}
and for all $i\in \{1,\dots,k\}$ we have
\begin{equation} \label{eq:f0b}
f_0(B(x_i,\delta))\subset B(y_i,r),
\end{equation}
where $y_i=f_0(x_i)$. Choose $0<2\eps<\delta$ such that the balls $B(x_{i},2\eps)$ are disjoint and let
$K_i=B(x_i,\eps)$ for all $i\in \{1,\dots,k\}$. As $\dim_T K_i \geq n$, by Lemma~\ref{l:stab} there exist onto maps
$g_i\colon K_i\to B(y_i,r)$ such that $B(y_i,r-t)\subset f_i(K_i)$ for all $0<t<r$ and $f_i\in B(g_i,t)$.
Then \eqref{eq:f0b} and Tietze's extension theorem yield that there are continuous maps
$G_i\colon B(x_i,\delta)\to B(y_i,r)$ such that $G_i=g_i$ on $K_i$ and $G_i=f_0$ on
$B(x_i,\delta)\setminus U(x_i,2\eps)$. Let $g(x)=G_i(x)$
for all $x\in B(x_i,\delta)$ and $i\in \{1,\dots,k\}$. Then the construction and \eqref{eq:K=} imply that
$g\in C(K,\RR^n)$ is well-defined and $g(B(x_i,\delta))=B(y_i,r)$ for all $i$. Therefore
\[ g(K)=\bigcup_{i=1}^{k} B(y_i,r). \]
The construction and \eqref{eq:f0b} imply that
\[ f_0(B(x_i,\delta))\cup g(B(x_i,\delta))=B(y_i,r), \]
which yields that $g\in B(f_0,2r)$.

Finally, assume that $0<t<r$ and $f\in B(g,t)$. Let $f_i=f|_{K_i}$, then clearly $f_i\in B(g_i,t)$,
so for all $i\in \{1,\dots,k\}$ we have
\[ B(y_i,r-t)\subset f_i(K_i)\subset f(K). \]
Hence $g\in \iG_n(K)$, and the proof is complete.
\end{proof}

\begin{lemma} \label{l:cn} For each $n\in \NN^+$ there is a finite constant $c_n$ such that
the set $B(y,r+\eps)\setminus U(y,r-\eps)$ can be covered by at most $c_n ((r+2\eps)/\eps)^{n-1}$ sets of diameter
$\eps$ for all $y\in \RR^n$ and $0<2\eps<r$.
\end{lemma}

\begin{proof} Let $n\in \NN^+$, $y\in \RR^n$, and $0<2\eps<r$ be arbitrarily fixed. Let
\[ R_1=B(y,r+\eps)\setminus U(y,r-\eps) \quad  \textrm{and} \quad R_2=B(y,r+2\eps)\setminus U(y,r-2\eps).\]
Define
\[ \delta=\frac{\eps}{\sqrt{n}} \quad \textrm{and} \quad
\iQ_{\delta}=\left\{\prod_{i=1}^n [k_i\delta, (k_i+1)\delta]: k_i\in \ZZ\right\}.
\]
Clearly
\begin{equation*}
(r+2\eps)^n-(r-2\eps)^n=4\eps \sum_{k=0}^{n-1} (r+2\eps)^{n-1-k} (r-2\eps)^{k}\leq 4\eps n (r+2\eps)^{n-1}.
\end{equation*}
For each $Q\in \iQ_{\delta}$ we have $\diam Q=\eps$, and if $Q\cap R_1\neq \emptyset$
then $Q\subset R_2$. Thus the above inequality yields that
the number of sets of diameter $\eps$ needed to cover $R_1$ is at most
\begin{align*}
\#\{Q\in \iQ_{\delta}: Q\cap R_1\neq \emptyset\}&\leq \#\{Q\in \iQ_{\delta}: Q\subset R_2\} \\
&\leq \delta^{-n} \iL^n(R_2) \\
&=\eps^{-n} n^{n/2} e_n((r+2\eps)^n-(r-2\eps)^n) \\
&\leq 4e_n n^{n/2+1}\left(\frac{r+2\eps}{\eps}\right)^{n-1} \\
&=c_n \left(\frac{r+2\eps}{\eps}\right)^{n-1},
\end{align*}
where $\iL^n$ denotes the $n$-dimensional Lebesgue measure, $e_n=\iL^n(B(\mathbf{0},1))$,
and $c_n=4e_n n^{n/2+1}$. The proof is complete.
\end{proof}

Now we are ready to prove Theorem~\ref{t:M}.

\begin{proof}[Proof of Theorem~\ref{t:M}] For all $g\in \iG_n(K)$ choose $k(g)\in \NN^+$ and $r(g)>0$
according to Definition~\ref{d:GnK}, and for each $m\in \NN^+$ let
$\varepsilon(g,m)$ be a small enough positive number to be chosen later such that
$\eps(g,m)<r(g)/2$ and $\lim_{m\to \infty} \eps(g,m)=0$. Define
\[ \iF=\bigcap_{m=1}^{\infty} \iF_m, \quad \textrm{where} \quad \iF_m=\bigcup_{g\in \iG_n(K)} U(g,\varepsilon(g,m)).\]
Lemma~\ref{l:GnK} yields that $\iF_m$ is a dense open set in $C(K,\RR^n)$ for every $m$, so $\iF$ is co-meager in $C(K,\RR^n)$.
We will prove that $\iH^h(\partial f(K))=0$ for each $f\in \iF$. Let us fix an arbitrary $f\in \iF$ and $m\in \NN^+$. Choose $g\in \iG_n(K)$ such that $f\in U(g,\varepsilon(g,m))$. It is enough to prove that $\eps=\eps(g,m)$ satisfies
\begin{equation} \label{eq:part} \iH^{h}_{\eps}(\partial f(K))\leq \frac 1m.
\end{equation}
Set $r=r(g)$ and $k=k(g)$. By the definition of $\iG_n(K)$ there are $y_1,\dots,y_k\in \RR^n$ such that $g(K)=\bigcup_{i=1}^k B(y_i,r)$ and 
\[\bigcup_{i=1}^{k} B(y_i,r-\varepsilon) \subset f(K)\subset \bigcup_{i=1}^{k} B(y_i,r+\varepsilon).\]
Thus
\begin{equation} \label{eq:gKU}
\partial f(K)\subset \bigcup_{i=1}^{k} (B(y_i,r+\varepsilon)\setminus U(y_i,r-\eps)).
\end{equation}
By Lemma~\ref{l:cn} there is a constant $c_n\in \RR^+$ such the sets
$B(y_i,r+\varepsilon)\setminus U(y_i,r-\eps)$ can be covered by at most
$c_n((r+2\eps)/\eps)^{n-1}$ sets of diameter $\eps$.
Now define $\eps$ such that $0<\eps<\min\{r/2, 1/m\}$ and
\begin{equation} \label{eq:eps}
\frac{h(\eps)}{\eps^{n-1}}\leq \frac{(r+2\eps)^{1-n}}{kc_n m}.
\end{equation}
Therefore
\eqref{eq:gKU} and \eqref{eq:eps} imply that
\[ \iH^{h}_{\eps} (\partial f(K))\leq kc_n\left(\frac{r+2\eps}{\eps}\right)^{n-1} h(\varepsilon)\leq \frac 1m,\]
so \eqref{eq:part} holds. The proof is complete.
\end{proof}

\section{Fibers of maximal dimension} \label{s:max}

The Main Theorem implies the following.

\begin{corollary} \label{c:maxtop} Assume that $K$ is a compact metric space and $n\in \NN^+$ such that $n\leq \dim_T K<\infty$.
Then for a generic $f\in C(K,\RR^n)$ there is a non-empty open set $U_f\subset \RR^n$ such that
for all $y\in U_f$ we have
\[ \dim_T f^{-1}(y)=d_{T}^n(K). \]
In particular, for a generic $f\in C(K,\RR^n)$ we have
\[ \max\{\dim_T f^{-1}(y): y\in \RR^n\}=d_T^{n}(K). \]
\end{corollary}

Thus the supremum is attained in Corollary~\ref{c:main1} if $\dim_{*}=\dim_T$ and $\dim_T K$ is finite.
We show that this is true in general by following the proof of \cite[Theorem~4.1]{BBE2}.

\begin{theorem} \label{t:max} Let $n\in \NN^+$ and let $K$ be a compact metric space with $\dim_T K\geq n$.
Let $\dim_{*}$ be one of $\dim_T$, $\dim_H$, or $\dim_P$. For a generic $f\in C(K,\RR^n)$ we have
\[ \max\{\dim_{*}f^{-1}(y): y\in \RR^n\}=d^n_{*}(K). \]
\end{theorem}

We need some lemmas for the proof, the next one is \cite[Lemma~7.2]{B}.

\begin{lemma} \label{l:max} Let $n\in \NN^+$ and let $K$ be a compact metric space with
$x_0\in K$. Let $K_m\subset K$ be compact sets such that
\begin{enumerate}[(i)]
\item  $\dim_{T}K_m\geq n$ for all $m\in \mathbb{N}^+$ and
\item \label{eq:ifi} $\diam \left(K_m\cup \{x_0\}\right)\to 0$ if $m\to \infty$.
\end{enumerate}
Then for a generic $f\in C(K,\RR^n)$ we have $f(x_0) \in f(K_m)$ for
infinitely many $m$.
\end{lemma}

\begin{lemma} \label{l:DD} Let $n\in \NN^+$ and let $K$ be a compact metric space with $\dim_T K\geq n$.
Let $\dim_{*}$ be one of $\dim_T$, $\dim_H$, or $\dim_P$.
Then there exists an $x_0\in K$ such that
for every $d<d^n_{*}(K)$ and $\eps>0$ there is a compact set $D\subset B(x_0,\eps)\setminus \{x_0\}$
such that $d^n_{*}(D\cap U)>d$ for all open sets $U \subset K$ intersecting $D$.
\end{lemma}

\begin{proof} If $d^n_{*}(K)=0$ then let $x_0\in K$ be an accumulation point of $K$ and let
$D=\{x_1\}$ such that $x_1\in B(x_0,\eps)\setminus \{x_0\}$.
Thus we may assume that $d^n_{*}(K)>0$.

First we prove that there is an $x_0\in K$ such that $d^n_{*}(U)=d^{n}_{*}(K)$ for all open sets $U\subset K$ containing $x_0$.
Assume to the contrary that for all $x\in K$ there is an open set $U_x\subset K$ such that $x\in U_x$ and $d^n_{*}(U_x)<d^n_{*}(K)$.
As $K$ is compact, $\{U_x: x\in K\}$ contains a finite cover $\{U_1,\dots, U_k\}$ of $K$.
Then $\bigcup_{i=1}^k U_i=K$ and $\sup\{d^n_{*}(U_i) : 1\leq i\leq k\}<d^{n}_{*}(K)$, which contradicts the countable
stability of $d^n_{*}$ for $F_{\sigma}$ sets.

Fix an $x_0$ as above. Assume that $d<d^n_{*}(K)$ and $\eps>0$ are given. Define $B_0=B(x_0,\eps)$,
then clearly $d^n_{*}(B_0)=d^n_{*}(K)$. Now we prove that there is a compact set $C\subset B_0\setminus \{x_0\}$ with
$d^n_{*}(C)>d$. Assume to the contrary that there is no such $C$, then the
sets $B_{i}=B(x_0,\eps)\setminus U(x_0,1/i)$ satisfy $d^n_{*}(B_i)\leq d$ for all $i\in \NN^+$.
Clearly $B_0=\bigcup_{i=1}^{\infty} B_i \cup \{x_0\}$, and
\[ \sup_{i\geq 1} d^n_{*}(B_i\cup \{x_0\})\leq \max\{d,0\}<d^n_{*}(K)=d^n_{*}(B_0), \]
which contradicts the
countable stability of $d^n_{*}$ for closed sets.

Finally, let $C\subset B(x_0,\eps) \setminus \{x_0\}$ be a compact set with
$d^n_{*}(C)>d$. It is enough to show that there is a compact set $D\subset C$ such that $d^n_{*}(D\cap U)>d$ for all open sets
$U\subset K$ intersecting $D$. Let $\iU$ be a countable open basis for $C$ and define
\[ D=C\setminus \bigcup \{U\in \iU: d^n_{*}(U)\leq d\}. \]
Clearly $D$ is compact, and the countable stability of $d^n_{*}$ for $F_{\sigma}$ sets yields that
 $d^n_{*}(C\setminus D)\leq d$. Assume to the contrary that there is an $U\in \iU$ intersecting $D$
 such that $d^{n}_{*}(D\cap U)\leq d$. Then clearly $d^n_{*}(U\setminus D)\leq d^n_{*}(C\setminus D)\leq d$, and
 the definition of $D$ yields that $d^n_{*}(U)>d$.
Therefore
\[ \max \{d^n_{*}(D\cap U),\, d^n_{*} (U\setminus D)\}\leq d<d^{n}_{*}(U), \]
which contradicts the countable stability of $d^n_{*}$ for $F_{\sigma}$ sets.
\end{proof}

The following lemma will be the heart of the proof of Theorem~\ref{t:max}.

\begin{lemma} \label{l:max2} Let $n\in \NN^+$ and let $D\subset K$ be compact metric spaces with $x_0\in K\setminus D$.
Assume that $d<d_{*}^{n}(D\cap U)$ for all open sets $U\subset K$ intersecting $D$.
Then for a generic $f\in C(K,\RR^n)$ either $\dim_{*}f^{-1}(f(x_0))\geq d$ or $f(x_0)\notin f(D)$.
\end{lemma}

\begin{proof} Clearly we may assume that $d>0$. We need to prove that the set
\[ \mathcal{F}=\left\{f\in C(K,\RR^n): \dim_{*}f^{-1}(f(x_0))\geq d \textrm{ or } f(x_0)\notin f(D) \right\} \]
is co-meager in $C(K,\RR^n)$. Consider
\[ \Gamma=\left\{(f,y)\in C(D, \RR^n)\times \mathbb{R}^n: \dim_{*}f^{-1}(y)\geq d
\textrm{ or } y\notin f(D)\right\}.\]
First assume that $\Gamma$ is co-meager in $C(D,\RR^n)\times \mathbb{R}^n$,
we prove that $\mathcal{F}\subset C(K,\RR^n)$ is also co-meager. Let
\[R\colon C(K,\RR^n)\to C(D,\RR^n)\times \mathbb{R}^n, \quad R(f)=(f|_{D},f(x_0)).\]
Clearly $R$ is continuous, and Tietze's extension theorem implies that it is open.
Thus Lemma~\ref{l:cat} yields that
$\mathcal{F}=R^{-1}(\Gamma)$ is co-meager.

Finally, we prove that $\Gamma$ is co-meager in $C(D,\RR^n)\times \mathbb{R}^n$.
Lemma~\ref{l:Baire} yields that $\Gamma$ is in $\sigma(\mathbf{A})$, so it has the Baire property.
Hence it is enough to prove by the Kuratowski-Ulam Theorem \cite[Theorem~8.41]{Ke}
that for a generic $f\in C(D,\RR^n)$ for a generic $y\in \mathbb{R}^n$ we have $(f,y)\in \Gamma$.
Let $\{z_i\}_{i\geq 1}$ be a dense set in $D$ and for all $i,j\in
\mathbb{N}^{+}$ define $B_{i,j}=D\cap B(z_i,1/j)$. For all $i,j$ let
\begin{align*} \mathcal{G}_{i,j}=\{&f\in C(B_{i,j},\RR^n): \textrm{ there is a non-empty open set } \\
&U_f\subset \RR^n \textrm{ such that } \dim_{*}f^{-1}(y) \geq d \textrm{ for all } y\in U_f\},
\end{align*}
and let
\[ R_{i,j}\colon C(D,\RR^n)\to C(B_{i,j},\RR^n), \quad R_{i,j}(f)=f|_{B_{i,j}}.\]
Define
\[ \mathcal{G}=\bigcap_{i,j\in \mathbb{N}^{+}} R_{i,j}^{-1} (\mathcal{G}_{i,j}). \]
Our condition yields $d^n_{*}(B_{i,j})>d$, so $\dim_T B_{i,j}\geq n$ by Fact~\ref{f:equiv}. Therefore the Main Theorem implies that
$\mathcal{G}_{i,j}$ are co-meager in $C(B_{i,j},\RR^n)$. Corollary~\ref{c:R} yields that
$R_{i,j}^{-1} (\mathcal{G}_{i,j})$ are co-meager in
$C(D,\RR^n)$, and as a countable intersection of co-meager sets $\mathcal{G}$ is also
co-meager in $C(D,\RR^n)$. Fix $f\in \mathcal{G}$. It is sufficient to verify
that $\Gamma_{f}=\{y\in \mathbb{R}^n: (f,y)\in \Gamma\}$ is co-meager. Let
$V\subset \mathbb{R}^n$ be an arbitrary non-empty open set, it is enough to
prove that $\Gamma_{f}\cap V$ contains a non-empty open set. We may assume that $f(D) \cap V\neq \emptyset$, otherwise
$V\subset \Gamma_f $ and we are done.
Then there exist $i,j\in \mathbb{N}^{+}$ such that $B=B_{i,j}$ satisfies $f(B)\subset V$.
The definition of $\mathcal{G}$ implies that there is a non-empty open set
$U_{f|_{B}}\subset V$ such that for all $y\in U_{f|_{B}}$ we
have
\[ \dim_{*}f^{-1}(y)\geq \dim_{*} (f|_{B})^{-1}(y)\geq d. \]
Hence $U_{f|_{B}} \subset \Gamma_{f}\cap V$, and the proof of the lemma is complete.
\end{proof}

Now we are able to prove Theorem~\ref{t:max}.

\begin{proof}[Proof of Theorem~\ref{t:max}]
By the Main Theorem it is enough to prove that for a generic
$f\in C(K,\RR^n)$ there exists a $y_f\in \RR^n$ such that
$\dim_{*} f^{-1}(y_f)\geq d_{*}^{n}(K)$. We may assume that $d^n_{*}(K)>0$, and
let $d_m$ be a positive sequence such that $d_m \nearrow d^n_{*}(K)$.
By Lemma~\ref{l:DD} there exists $x_0\in K$ such that for each $m\geq 1$ there is a compact set
$K_m\subset B(x_0,1/m)\setminus \{x_0\}$ such that $d^n_{*}(K_m\cap U)>d_m$ for all open sets $U\subset K$ intersecting
$K_m$. As $d^{n}_{*}(K_m)\geq d_m>0$, Fact~\ref{f:equiv} yields $\dim_T K_m\geq n$.
Therefore we can apply Lemma~\ref{l:max} for the sequence $\{ K_m\}_{m\geq 1}$ in $K$,
and Lemma~\ref{l:max2} for all $K_m\subset K$ with $d_m$. These imply that for a generic $f\in C(K,\RR^n)$ we have $f(x_0) \in f(K_m)$
for infinitely many $m$, and for every $m$ either $d^n_{*} (f^{-1}(f(x_0)))\geq d_m$ or $f(x_0)\notin f(K_m)$.
Hence there is a strictly increasing sequence $\{ m_i \}_{i \geq 1}$ depending on $f$ such that
$d^n_{*}(f^{-1}(f(x_0)))\geq d_{m_i}$
for all $i$, that is,
\[ d^n_{*}(f^{-1}(f(x_0)))\geq \sup\{d_{m_i}: i\geq 1\}=d^n_{*}(K).\]
This concludes the proof.
\end{proof}

Assume that $\dim_{*}$ is one of $\dim_H$ or $\dim_P$, or $\dim_{*}=\dim_T$ and $\dim_T K=\infty$.
Then we prove that Theorem~\ref{t:max} is best possible, that is, there is exactly one fiber of maximal dimension.

\begin{theorem} \label{t:ex} For each $n\in \NN^+$ there is a compact set $K\subset \RR^{n+1}$ such that
for each $f\in C(K,\RR^n)$ there is a $y_f\in \RR^n$ such that
\begin{enumerate}[(i)]
\item \label{HP1} $d_H^n(K)=1$ and $d_{P}^{n}(K)=n+1$,
\item \label{HP2} $\dim_H f^{-1}(y)<1$ for a generic $f\in C(K,\RR^n)$ for all $y\in \RR^n\setminus \{y_f\}$,
\item \label{HP3} $\dim_P f^{-1}(y)<n+1$ for every $f\in C(K,\RR^n)$ for all $y\in \RR^n\setminus \{y_f\}$.
\end{enumerate}
\end{theorem}

\begin{proof} For each $i\in \NN^+$ let $C_i\subset [0,1/i]$ be compact sets such that $0\in C_1$ and
\[ \dim_H C_i=\dim_P C_i=1-1/i. \]
Define
\[ K=\bigcup_{i=1}^{\infty} K_i, \textrm{ where } K_i=[0,1/i]^n\times C_i. \]
Let $\mathbf{0}$ denote the origin of $\RR^{n+1}$, then $\mathbf{0}\in K$ yields that $K$ is compact.
Since $d_H^n$ and $d_P^n$ are countably stable for closed sets, Lemmas~\ref{l:prodH} and \ref{l:prodP} imply that
\begin{align*}
d_{H}^n(K)&=\sup_{i\geq 1} d^n_{H}(K_i)=\sup_{i\geq 1} \dim_H C_i=1, \\
d_{P}^n(K)&=\sup_{i\geq 1} d^n_{P}(K_i)=\sup_{i\geq 1} \dim_P C_i+n=n+1.
\end{align*}
Therefore \eqref{HP1} holds. 

For each $f\in C(K,\RR^n)$ let $y_f=f(\mathbf{0})$.
For all $k\in \NN^+$ let $D_k=\bigcup_{i=1}^k K_i$ and let
\[ \iF_k=\{f\in C(D_k,\RR^n): \dim_H f^{-1}(y)\leq d^n_H(D_k) \textrm{ for all } y\in \RR^n\}. \]
For each $k\in \NN^+$ define
\[ R_k\colon C(K,\RR^n)\to C(D_k,\RR^n), \quad R_k(f)=f|_{D_k}. \]
Finally, let
\[ \iF=\bigcap_{k=1}^{\infty} R_k^{-1}(\iF_k).\] 
By the Main Theorem $\iF_k\subset C(D_k,\RR^n)$ are co-meager, and
Corollary~\ref{c:R} implies that $R_k^{-1}(\iF_k)\subset C(K,\RR^n)$ are co-meager as well.
As a countable intersection of co-meager sets, $\iF\subset C(K,\RR^n)$ is also co-meager. Let $f\in \iF$ and $y\in \RR^n\setminus \{y_f\}$,
we prove that $\dim_H f^{-1}(y)<1$. Indeed, there is a $k=k(f,y)\in \NN^+$ such that $f^{-1}(y)\subset D_k$. Thus the
definition of $\iF$, the countable stability of $d^n_H$, and Lemma~\ref{l:prodH} imply that
\[ \dim_H f^{-1}(y)\leq d^n_{H}(D_{k})=\sup_{i\leq k}d^{n}_{H}(K_i)=\sup_{i\leq k} \dim_H C_i=1-1/k<1. \]
Thus \eqref{HP2} holds. 

The countable stability of packing dimension and Lemma~\ref{l:dimp} imply that
\[ \dim_P D_k=\sup_{i\leq k}\dim_P K_i=\sup_{i\leq k} \dim_P C_i+n=1-1/k+n<n+1. \]
Finally, let $f\in C(K,\RR^n)$ and $y\in \RR^n\setminus \{y_f\}$ be arbitrary. Then there exists a
$k=k(f,y)\in \NN^+$ such that $f^{-1}(y)\subset D_k$, so
\[ \dim_P f^{-1}(y)\leq \dim_P D_k<n+1. \]
Hence \eqref{HP3} holds, and the proof is complete.
\end{proof}

\begin{fact} \label{f:ex2} There is a compact metric space $K$ such that
$\dim_T K=\infty$, and for each $f\in C(K,\RR^n)$ there is a
$y_f\in \RR^n$ such that $\dim_{T} f^{-1}(y)<\infty$ for all $y\in \RR^n\setminus \{y_f\}$.
\end{fact}

\begin{proof} Let $[0,1]^{\NN}$ be the Hilbert cube endowed with a complete metric compatible with the product topology,
and let $\mathbf{0}\in [0,1]^{\NN}$ be the zero sequence. Let us define $K\subset [0,1]^{\NN}$ as
\[K=\bigcup_{i=1}^{\infty} K_i, \textrm{ where } K_i=[0,1/i]^{i}\times \{\mathbf{0}\}.\]
As $\mathbf{0}\in K$, the set $K$ is compact. For all $f\in C(K,\RR^n)$ and $k\in \NN^+$ let $y_f=f(\mathbf{0})$ and
$D_k=\bigcup_{i=1}^k K_i$. Fix $f\in C(K,\RR^n)$ and $y\in \RR^n\setminus \{y_f\}$. Then there exists a $k=k(f,y)\in \NN^+$ such that $f^{-1}(y)\subset D_k$. The countable stability of topological dimension for closed sets yields that
\[ \dim_T f^{-1}(y)\leq \dim_T D_k=\sup_{i\leq k} \dim_T K_i=k<\infty. \]
The proof is complete.
\end{proof}

\subsection*{Acknowledgments}
The author is grateful to Zolt\'an Buczolich and M\'arton Elekes for their helpful suggestions.

\end{document}